\documentclass{amsart}
\usepackage{amsfonts,amssymb,amsmath,amsthm}
\usepackage{url}
\usepackage{enumerate}

\usepackage[francais]{babel} 

\newtheorem{thrm}{Th\'eor\`eme}[section]
\newtheorem{lemme}[thrm]{Lemme}
\newtheorem{proposition}[thrm]{Proposition}
\newtheorem{corollaire}[thrm]{Corollaire}
\theoremstyle{definition}
\newtheorem{definition}[thrm]{D\'efinition}
\newtheorem{remarque}[thrm]{Remarque}
\numberwithin{equation}{section}

\DeclareMathOperator{\ric}{Ric}
\DeclareMathOperator{\scal}{Scal}
\DeclareMathOperator{\R}{R}
\DeclareMathOperator{\Se}{S}
\DeclareMathOperator{\W}{W}
\DeclareMathOperator{\Sc}{\it{P}}
\DeclareMathOperator{\B}{\it{B}}
\DeclareMathOperator{\Pa}{P}
\DeclareMathOperator{\tr}{tr}

\author{Vincent B\'erard}
\address{Institut de Math\'ematiques et Mod\'elisation de Montpellier \\
UMR 5149 CNRS - Universit\'e Montpellier II}
\email{vberard@math.univ-montp2.fr}

\begin{document}

\title[Les applications c--harmoniques]{Les applications conforme--harmoniques}

\begin{abstract}
Sur une surface de Riemann, l'\'energie d'une application \`a valeurs dans une vari\'et\'e riemannienne est une fonctionnelle invariante conforme, ses points critiques sont les applications harmoniques. Nous proposons ici un analogue en dimension sup\'erieure, en construisant une fonctionnelle invariante conforme pour les applications entre deux vari\'et\'es riemanniennes, dont la vari\'et\'e de d\'epart est de dimension $n$ paire. Ses points critiques satisfont une EDP elliptique d'ordre $n$ non--lin\'eaire qui est covariante conforme par rapport \`a la vari\'et\'e de d\'epart, on les appelle les applications conforme--harmoniques. Dans le cas des fonctions, on retrouve l'op\'erateur GJMS, dont le terme principal est une puissance $n/2$ du laplacien. Quand $n$ est impaire, les m\^emes id\'ees permettent de montrer que le terme constant dans le d\'eveloppement asymptotique de l'\'energie d'une application asymptotiquement harmonique sur une vari\'et\'e AHE est ind\'ependant du choix du repr\'esentant de l'infini conforme.
\end{abstract}

\maketitle

%%%%%%%%%%%%%%%%%%%%%%%%%%%%%%%%%%%%%%%%%%%%%%%%%%%%%%%%%%%%%%%%%%%%%%%%%%%%%%%%%%%%
%%%%%%%%%%%%%%%%%%%%%%%%%%%%%%%%%%%%%%%%%%%%%%%%%%%%%%%%%%%%%%%%%%%%%%%%%%%%%%%%%%%%
\section{Introduction}
%%%%%%%%%%%%%%%%%%%%%%%%%%%%%%%%%%%%%%%%%%%%%%%%%%%%%%%%%%%%%%%%%%%%%%%%%%%%%%%%%%%%
%%%%%%%%%%%%%%%%%%%%%%%%%%%%%%%%%%%%%%%%%%%%%%%%%%%%%%%%%%%%%%%%%%%%%%%%%%%%%%%%%%%%

Soient $(M,g)$ et $(N,h)$ deux vari\'et\'es riemanniennes de dimension $n$ et $m$, dans toute la suite, on consid\'erera que ces vari\'et\'es sont compactes et de classe $C^{\infty}$. On appelle \'energie des applications de $(M,g)$ dans $(N,h)$, la fonctionnelle $E_g$ d\'efinie de la mani\`ere suivante\,:
\begin{equation*}
 E_g(\varphi) = \frac{1}{2} \int_M |T\varphi|^2_{g,h} \,dvol_g,
\end{equation*}
o\`u $T\varphi$ d\'esigne l'application tangente de $\varphi$, qui est une section du fibr\'e des $1$--formes \`a valeurs dans les champs de vecteurs de $TN$ tir\'es--en--arri\`ere par $\varphi$, qu'on note $\Omega^1(M) \otimes \varphi^* TN$. Les applications harmoniques de $(M,g)$ dans $(N,h)$ sont d\'efinies comme les points critiques de l'\'energie et un r\'esultat classique les caract\'erise comme \'etant les solutions de l'\'equation $\delta^g T\varphi=0$, o\`u $\delta^g$ d\'esigne la divergence du fibr\'e $\Omega^1(M) \otimes \varphi^* TN$ construit canoniquement avec les connexions de Levi--Civita de $g$ et $h$ (celle de $h$ \'etant tir\'ee--en--arri\`ere par $\varphi$). Dans le cas des applications d'une surface \`a valeurs dans une vari\'et\'e riemannienne quelconque, il est connu que l'\'energie ne d\'epend que de la classe conforme de la m\'etrique de d\'epart (et bien sur de l'application et de la m\'etrique d'arriv\'ee), c'est--\`a--dire que pour deux m\'etriques conformes $g$ et $\overline{g}:=e^{2\omega}g$ sur une vari\'et\'e de dimension $2$, on a\,:
\begin{center}
 $E_{\overline{g}}(\varphi) = E_{g}(\varphi)$ et $\delta^{\overline{g}} T\varphi = e^{-2\omega} \delta^g T\varphi$.
\end{center}
Il faut remarquer que le laplacien est en g\'en\'eral un op\'erateur non--lin\'eaire, ainsi quand la vari\'et\'e $M$ est une surface, \^etre harmonique signifie que l'application en question est solution d'une \'equation non--lin\'eaire d'ordre $2$ qui est covariante par changement conforme de m\'etrique sur la surface. Par contre ce n'est plus le cas quand $M$ est de dimension strictement sup\'erieure \`a $2$, l'\'energie n'est plus un invariant conforme et on obtient alors \cite[1.159.i)]{MR867684}\,:
\begin{equation*}
 \delta^{\overline{g}} T\varphi = e^{-2\omega} \big(\delta^g T\varphi - (n-2) \langle d\omega,T\varphi \rangle_g \big).
\end{equation*}
L'harmonicit\'e n'est plus une propri\'et\'e g\'eom\'etrique de la classe conforme de $g$, mais bien de la m\'etrique $g$. 

On peut trouver dans la litt\'erature (voir \cite{MR1692148}, \cite{MR2301373} et les r\'ef\'erences cit\'ees) une g\'en\'eralisation des applications harmoniques qui est non--conforme, ce sont les applications biharmoniques, qui sont d\'efinies comme \'etant les points critiques de la bi\'energie\,:
\begin{equation*}
 E^2_g(\varphi) := \frac{1}{2} \int_M |\delta^g T\varphi|_g^2 \, dvol_g.
\end{equation*}
Il s'agit d'une classe d'applications qui englobent les applications harmoniques et qui permet, comme par exemple dans \cite{MR2124627}, de donner une nouvelle d\'emonstration du th\'eor\`eme d'Eells--Sampson sur l'existence d'applications harmoniques dans les classes d'homotopie. On sait que les applications harmoniques n'existent pas toujours (voir l'article \cite{MR0420708} d'Eells et Wood) et un des principaux objectifs de cette th\'eorie est de vouloir prouver l'existence d'applications biharmoniques dans ces cas l\`a. On peut citer encore les travaux de Baird et de Kamissoko dans \cite{MR1952859}, qui utilisent justement le fait que l'harmonicit\'e ne soit pas une notion invariante conforme en dimension sup\'erieure \`a $2$, pour exhiber des applications biharmoniques qui ne sont pas harmoniques. Nous nous poserons le m\^eme type de question et nous obtiendrons aussi un r\'esultat d'existence pour notre nouvelle classe d'applications.

\vskip 0.2cm

Le but de cet article est de d\'efinir une nouvelle notion d'harmonicit\'e pour les applications sur les vari\'et\'es de dimension paire qui soit invariante conforme, c'est--\`a--dire d\'efinir une fonctionnelle invariante conforme qui va jouer le r\^ole de l'\'energie et d\'eterminer l'\'equation de ses points critiques qui va remplacer la condition non--lin\'eaire d'annulation du laplacien. Si on se restreint aux fonctions sur les vari\'et\'es de dimension paire, Graham, Jenne, Mason et Sparling, ont d\'emontr\'e en 1987 dans \cite{MR1190438}, l'existence d'un op\'erateur diff\'erentiel covariant conforme de terme principal $\Delta^{n/2}$ sur les fonctions $C^{\infty}$ de $M$. En dimension $4$, il s'agit de l'op\'erateur de Paneitz $\Pa_4$\,:
\begin{equation*}
 \Pa_4 := \Delta^2 + \delta (\frac{2}{3}\scal - 2\ric)\, d.
\end{equation*}
Nous proposons de g\'en\'eraliser l'\'equation du noyau de cet op\'erateur sur des fonctions, en une \'equation aux d\'eriv\'ees partielles elliptique non--lin\'eaire d'ordre $n$ sur les applications $C^{\infty}$ de $(M,g)$ dans $(N,h)$ qui soit covariante conforme par rapport \`a $g$. De plus, nous construisons une fonctionnelle invariante conforme par rapport \`a $g$, dont les points critiques sont exactement les solutions de cette EDP. Bien que la d\'emonstration de l'existence de cette EDP suit les id\'ees de Graham, Jenne, Mason et Sparling en r\'esolvant un probl\`eme de Cauchy, la fonctionnelle s'obtient en renormalisant l'\'energie de la solution de ce probl\`eme \`a bord, en suivant l'id\'ee de Graham dans \cite{MR1758076} quand il d\'efinit son volume renormalis\'e. Nous obtenons le th\'eor\`eme suivant qui r\'esume les th\'eor\`emes \ref{theoreme_phi_pair} et \ref{theoreme_fonctionnelle_paire}\,:
\begin{thrm}
 Soit $(M^n,g)$ et $(N,h)$ deux vari\'et\'es riemanniennes, on suppose que $n$ est pair, alors il existe une fonctionnelle sur les applications de classe $C^{\infty}$ de $(M,g)$ dans $(N,h)$ qui est invariante conforme par rapport \`a $g$. De plus, l'\'equation de ses points critiques est une \'equation aux d\'eriv\'ees partielles elliptique non--lin\'eaire d'ordre $n$, qui est covariante conforme elle aussi par rapport \`a $g$.
\end{thrm}
Nous d\'efinissons les applications conforme--harmoniques de la mani\`ere suivante\,:
\begin{definition}
 On note $\mathcal{E}_g^n$ la fonctionnelle du th\'eor\`eme pr\'ec\'edent et on appelle ses points critiques, les applications conforme--harmoniques, qu'on abr\`ege en parlant d'applications C--harmoniques.
\end{definition}

On consid\`ere la vari\'et\'e $(M^n,g)$ comme \'etant l'infini conforme d'une vari\'et\'e $(X^{n+1},g_+)$ particuli\`ere. Il s'agit d'une g\'en\'eralisation du mod\`ele du disque de Poincar\'e, o\`u $(M,g)$ joue le r\^ole de la sph\`ere $\mathbb{S}^n$ munie de sa m\'etrique canonique et $(X,g_+)$ le r\^ole de la boule unit\'e de $\mathbb{R}^{n+1}$ munie de la m\'etrique hyperbolique, ce qui justifiera l'appellation m\'etrique de Poincar\'e de $(M,g)$ quand on parlera de $(X,g_+)$. L'\'equation des points critiques est obtenue comme une obstruction \`a r\'esoudre un probl\`eme de Cauchy d\'eg\'en\'er\'e sur $(X,g_+)$. On se donne une application $\varphi$ de $M$ dans $N$, il s'agit de d\'eterminer une application $\tilde{\varphi}$ qui soit $C^{\infty}$ de $\overline{X}$ dans $N$ qui v\'erifie les syst\`eme suivant\,:
\begin{center}
 $\left\{\begin{array}{ll}
  \tilde{\varphi}              & = \ \varphi \mbox{\ \ sur\ \ } M,\\
  \delta^{g_+}T\tilde{\varphi} & = \ 0       \mbox{\, \ sur\ \ } \overline{X}.
 \end{array}\right.$
\end{center} 
Sur les fonctions, Graham, Jenne, Mason et Sparling ont montr\'e qu'il n'\'etait pas toujours possible de r\'esoudre ce probl\`eme localement; quand la vari\'et\'e de d\'epart est de dimension paire, il existe un terme logarithmique non--trivial dans le d\'eveloppement formel de la solution pr\`es du bord qui obstrue la r\'egularit\'e de la r\'esolution. Ce terme est alors d\'efini comme l'op\'erateur GJMS d'ordre maximal en $\varphi$ qui ne d\'epend que de la classe conforme de $g$ et de la m\'etrique $h$. On va suivre la m\^eme id\'ee pour les applications, en identifiant localement la vari\'et\'e d'arriv\'ee avec son espace tangent, de mani\`ere \`a calculer le d\'eveloppement asymptotique de la compos\'ee de $\tilde{\varphi}$ avec l'exponentielle. Quand $M$ est de dimension paire, il y a un terme logarithmique qui appara\^it et qui ne d\'epend que de l'application de d\'epart et de la classe conforme de $g$ (et de la m\'etrique $h$), notre EDP est simplement la condition sur $\varphi$ que ce terme soit nul. Le d\'eveloppement asymptotique est enti\`erement d\'etermin\'e jusque--l\`a par $\varphi$ et des termes de courbures de nos deux vari\'et\'es $(M,g)$ et $(N,h)$, ce qui est \'equivalent \`a la donn\'ee de la valeur sur le bord des $n$ premi\`eres d\'eriv\'ees de la solution $\tilde{\varphi}$ par rapport \`a la coordonn\'ee radiale. Cela nous permet de calculer le d\'eveloppement asymptotique de l'\'energie dans un ruban $M \times [\rho;\varepsilon]$ quand $\rho$ tend vers $0$ qui admet un terme constant qui ne d\'epend que de $\varphi$ et de la classe conforme de $g$ (et de la m\'etrique $h$). On d\'efinit ce terme constant comme \'etant l'image de $\varphi$ par notre fonctionnelle et on montre ensuite, par une int\'egration par parties, que le gradient de cette fonctionnelle est bien le terme logarithmique pr\'ec\'edent.
 
\vskip 0.2cm

Du th\'eor\`eme pr\'ec\'edent, on obtient directement le r\'esultat suivant de rigidit\'e pour les applications harmoniques sur les boules de dimension impaire, qui se g\'en\'eralise au cas des vari\'et\'es asymptotiquement hyperboliques de dimension impaire (voir corollaire \ref{corollaire_rigidite_harmonique_bord}).
\begin{corollaire}
 On consid\`ere $(B,g_{hyp})$ la boule unit\'e ouverte de $\mathbb{R}^{n+1}$ de dimension impaire munie de la m\'etrique hyperbolique, $(S,[g_{can}])$ son infini conforme et $(N,h)$ une vari\'et\'e riemannienne, alors les applications qui sont de classe $C^n$ de $\overline{B}$ dans $N$ et harmonique de $(B,g_{hyp})$ dans $(N,h)$, v\'erifient le fait que leurs restrictions \`a $S$ est C--harmonique de $(S,[g_{can}])$ dans $(N,h)$.
\end{corollaire}

En dimension $2$, la fonctionnelle $\mathcal{E}_g^2$ est \'evidemment l'\'energie des applications de $(M^2,g)$ dans $(N,h)$ et les applications C--harmoniques sont exactement les applications harmoniques. En dimension $4$, le th\'eor\`eme \ref{theoreme-fonctionnelle-dimension4} nous donne une expression explicite de la fonctionnelle $\mathcal{E}_g^4$ en terme de courbures de $g$ et de l'\'equation de ses points critiques en termes de courbures de $g$ et de $h$\,:
\begin{thrm}
 Quand $M$ est de dimension $4$,
 \begin{equation*}
 \mathcal{E}_g^4(\varphi) = \int_M \big( |\delta T\varphi|^2_h + \frac{2}{3}\scal|T\varphi|^2_{g,h}
    - 2\,\ric (T\varphi,T\varphi) \big) \,dvol,
 \end{equation*}
 o\`u $\scal$ et $\ric$ d\'esignent respectivement la courbure scalaire et le tenseur de Ricci de $g$ et $dvol$ la forme volume de $g$. Notons $\Se$ l'endomorphisme de $\varphi^* TN$ d\'efini de la mani\`ere suivante\,:
 \begin{equation*}
  \Se(X) = \sum_{i=1}^4 \R^h_{\textstyle{X,T\varphi(e_i)}} T\varphi(e_i),
 \end{equation*}
 o\`u $(e_1,\ldots,e_4)$ est une base orthonorm\'ee de $TM$ par rapport \`a $g$ et $\R^h$ est le tenseur de courbure de $(N,h)$, alors l'\'equation de ses points critiques s'\'ecrit\,:
 \begin{equation*}
  \delta d\delta T\varphi + \delta \big((\frac{2}{3}\scal - 2\ric)\, T\varphi\big) - \Se(\delta T\varphi) = 0.
 \end{equation*}
\end{thrm}
Quand $M$ est de dimension $6$, la condition de C--harmonicit\'e et la fonctionnelle $\mathcal{E}_g^6$ sont explicit\'ees dans le th\'eor\`eme \ref{theoreme-fonctionnelle-dimension6-einstein} sous certaines conditions de courbures de nos deux vari\'et\'es. Sans ces hypoth\`eses, les calculs deviennent rapidement compliqu\'es et on est amen\'e alors \`a faire des hypoth\`eses sur l'application, comme par exemple regarder simplement l'identit\'e (voir th\'eor\`eme \ref{theoreme-fonctionnelle-dimension6-id}). 

\vskip 0.2cm

Si la vari\'et\'e de d\'epart est une vari\'et\'e d'Einstein de dimension paire quelconque, l'expression de sa m\'etrique de Poincar\'e est simple, la proposition \ref{proposition_einstein_C-harmonique} montre alors que les applications harmoniques sont C--harmoniques et la proposition \ref{proposition_fonctionnelle_einstein_harmonique} calcule explicitement leurs images par notre fonctionnelle. Cependant, la condition de C--harmonicit\'e reste encore compliqu\'ee \`a obtenir. 
Quand $M$ est dimension $4$, il existe certaines conditions de courbures sur nos deux vari\'et\'es, pour lesquelles les applications C--harmoniques sont alors exactement les applications harmoniques (voir la proposition \ref{proposition_harmonique=C-harmonique_dim4}), qui sont alors exactement les applications totalement g\'eod\'esiques, d'apr\`es une proposition due \`a Eells et Sampson dans \cite{MR0164306}.
L'identit\'e est toujours une application harmonique de $(M,g)$ dans $(M,g)$, ainsi elle est C--harmonique quand $M$ est de dimension $2$. Cela n'est plus le cas en dimension sup\'erieure, en dimension $4$, l'identit\'e est C--harmonique si et seulement si $g$ est \`a courbure constante (voir corollaire \ref{corollaire_id_dim4}).
Ce r\'esultat nous sert de point de d\'epart \`a la construction d'une application C--harmonique de $(M,[g])$ dans $(N,h)$, qui ne soit pas trivialement harmonique, c'est--\`a--dire qui ne soit pas harmonique de $(M,\overline{g})$ dans $(N,h)$, pour n'importe quelle m\'etrique $\overline{g}$ dans la classe conforme de $g$.
On se donne une vari\'et\'e $M^4$ munie d'une m\'etrique $h$ \`a courbure scalaire constante n\'egative proche d'une m\'etrique d'Einstein et on regarde les applications de $M$ dans $M$.
On va montrer que fixer la m\'etrique $h$ dans la vari\'et\'e d'arriv\'ee et d\'eformer judicieusement la m\'etrique de la vari\'et\'e de d\'epart, permet de construire une application proche de l'identit\'e qui conserve la C--harmonicit\'e, mais qui n'est plus harmonique, pour n'importe quel changement conforme de m\'etrique par rapport \`a la vari\'et\'e de d\'epart.
On obtient le th\'eor\`eme suivant (on pourra consulter le th\'eor\`eme \ref{theoreme_non_triviale} pour avoir un \'enonc\'e plus pr\'ecis)\,:

\begin{thrm}
 Soit $(M,g_e)$ une vari\'et\'e d'Einstein de dimension $4$ \`a courbure scalaire n\'egative, alors pour toute m\'etrique $h$ suffisamment proche de $g_e$ il existe une application $\varphi$ de $M$ dans $M$ et deux m\'etriques $g$ et $h$ proches de $g_e$ telles que\,:
 \begin{enumerate}
 \item $\varphi$ est C--harmonique de $(M,[g])$ dans $(M,h)$,
 \item $\varphi$ est non--harmonique de $(M,\overline{g})$ dans $(M,h)$, $\forall\overline{g} \in [g]$.
 \end{enumerate} 
\end{thrm}

Quand $n$ est impair, on a toujours une notion de m\'etrique de Poincar\'e de $(M,[g])$ et les premiers termes du d\'eveloppement formel de la solution de notre probl\`eme de Cauchy sont encore enti\`erement d\'etermin\'es pas nos conditions initiales, mais on a plus d'obstruction sous la forme d'un terme logarithmique et donc \`a priori plus de terme covariant conforme pour construire notre fonctionnelle.
De plus, le terme constant dans le d\'eveloppement asymptotique de l'\'energie d'une solution de notre probl\`eme \`a bord d\'epend alors de la connaissance de toute la solution et plus seulement de sa valeur au bord comme dans le cas pair, ce qui nous oblige \`a travailler avec des applications d\'efinies sur des vari\'et\'es asymptotiquement hyperbolique de bord \`a l'infini $(M,[g])$.
On obtient alors comme r\'esultat que le terme constant dans le d\'eveloppement asymptotique de l'\'energie d'une application asymptotiquement harmonique d'une vari\'et\'e asymptotiquement hyperbolique d'Einstein $(X,g_+)$ dans une vari\'et\'e riemannienne est ind\'ependant du choix de la m\'etrique dans l'infini conforme de $(X,g_+)$.
En outre, la variation infinit\'esimale de cette \'energie renormalis\'ee ne d\'epend que du premier terme du d\'eveloppement asymptotique de la solution qui d\'epend aussi de l'int\'erieur de la vari\'et\'e $X$ (voir le th\'eor\`eme \ref{theoreme_fonctionnelle_impaire} pour plus de d\'etails).

Afin de compl\'eter ces r\'esultats d'existence, citons le th\'eor\`eme obtenu r\'ecemment par Biquard et Madani dans \cite{2011arXiv1112.6130B}. Il s'agit d'un analogue conforme en dimension $4$, d'un c\'el\`ebre th\'eor\`eme d'Eells et Sampson (voir \cite{MR0164306}). Sous certaines hypoth\`eses de courbures de $(M^4,g)$ et $(N,h)$, ils prouvent l'existence d'une application C--harmonique dans chaque classe d'homotopie de $C^{\infty}(M,N)$.
\vskip 0.2cm

Cet article apporte les preuves et compl\`ete les r\'esultats annonc\'es dans la note parue aux Comptes Rendus Math\'ematiques de l'Acad\'emie des Sciences 
\cite{MR2449641}.

%%%%%%%%%%%%%%%%%%%%%%%%%%%%%%%%%%%%%%%%%%%%%%%%%%%%%%%%%%%%%%%%%%%%%%%%%%%%%%%%%%%%%%%%
%%%%%%%%%%%%%%%%%%%%%%%%%%%%%%%%%%%%%%%%%%%%%%%%%%%%%%%%%%%%%%%%%%%%%%%%%%%%%%%%%%%%%%%%
\section{La m\'etrique de Poincar\'e}
%%%%%%%%%%%%%%%%%%%%%%%%%%%%%%%%%%%%%%%%%%%%%%%%%%%%%%%%%%%%%%%%%%%%%%%%%%%%%%%%%%%%%%%%
%%%%%%%%%%%%%%%%%%%%%%%%%%%%%%%%%%%%%%%%%%%%%%%%%%%%%%%%%%%%%%%%%%%%%%%%%%%%%%%%%%%%%%%%

Soit $(X^{n+1},g_+)$ une vari\'et\'e non--compacte, on note $M$ le bord de son adh\'erence et on appelle fonction g\'eod\'esique d\'efinissant le bord de $X$, toute fonction $r$ de $\overline{X}$ v\'erifiant $r=0$ sur $M$, $r>0$ sur $X$ et $dr\neq0$ sur $M$. Notre m\'etrique $g_+$ est dite asymptotiquement hyperbolique (AH), s'il existe une fonction g\'eod\'esique $r$ telle que la m\'etrique $r^2g_+$ se prolonge en une m\'etrique non--d\'eg\'en\'er\'ee sur $\overline{X}$ et si ses courbures sectionnelles tendent vers $-1$ \`a l'infini. Il est facile de voir que cette derni\`ere hypoth\`ese est \'equivalente \`a $|dr|_{r^2g_+}=1$ sur $M$, qui est une condition qui d\'epend seulement de $g$ et pas du choix de la fonction g\'eod\'esique $r$. La donn\'ee d'une telle m\'etrique d\'etermine une classe conforme de m\'etrique sur le bord appel\'e infini conforme. Avec nos notations, l'infini conforme de $(X,g_+)$ est la classe conforme de la m\'etrique $r^2g_+$ restreinte \`a $TM$. Un th\'eor\`eme de Graham (\cite{MR1758076}) permet d'associer \`a chaque m\'etrique $g$ dans l'infini conforme d'une vari\'et\'e AH $(X,g_+)$, une unique (dans un voisinage de $M$) fonction g\'eod\'esique $r$ v\'erifiant\,:
\begin{equation}
 \label{forme normale}
 g_+ = \frac{dr^2 + g_r}{r^2},
\end{equation}
o\`u $g_r$ est une famille \`a $1$--param\`etre de m\'etriques sur $\partial\overline{X}$ v\'erifiant $g_0=g$.
Ce r\'esultat est le premier pas vers une g\'en\'eralisation du mod\`ele du disque de Poincar\'e, en vue de d\'eterminer une corr\'elation entre la g\'eom\'etrie de l'int\'erieur d'une vari\'et\'e et la g\'eom\'etrie conforme de son bord. Cependant, les \'equations sont trop souples et on rigidifie la situation en prenant une m\'etrique AH qui soit d'Einstein (AHE). Fefferman et Graham ont obtenu le th\'eor\`eme suivant\,:
\begin{thrm}[Fefferman--Graham \cite{MR837196}]
 \label{theoreme_g+}
 On se donne $(X,g_+)$ une vari\'et\'e AHE de dimension $n+1$ de bord \`a l'infini $M$, $g$ un repr\'esentant de son infini conforme et on \'ecrit $g_+$ sous la forme (\ref{forme normale}). Alors $g_r$ admet le d\'eveloppement asymptotique en $r=0$ suivant si $n$ est pair\,:
 \begin{equation}
  g_r = g + g_{(2)} r^2 + \cdots + g_{(n-2)} r^{n-2} + h \, r^n \log{r} + g_{(n)} r^n + O(r^{n+1}),
 \end{equation}
 et le d\'eveloppement suivant si $n$ est impair\,:
 \begin{equation}
  g_r = g + g_{(2)} r^2 + \cdots + g^{(n-1)} r^{n-1} + g_{(n)} r^n + O(r^{n+1}).
 \end{equation}
 Ces d\'eveloppements sont compos\'es de termes en puissance paires de $r$ jusqu'\`a l'ordre $n$ et un terme logarithmique $h$ dans le cas pair qui sont uniquement d\'etermin\'es par des termes de courbures de $g$, ainsi que la trace de $g_{(n)}$ par rapport \`a $g$ (elle est m\^eme nulle si $n$ est impair). De plus $h$ ne d\'epend que de la classe conforme de $g$.
\end{thrm}

\begin{remarque}
 Le terme $h$ dans le d\'eveloppement asymptotique de $g_r$ est, \`a un coefficient multiplicatif pr\`es, le tenseur d'obstruction de Graham et  Hirachi \cite{MR2160867}.
\end{remarque}
L'exemple de base est bien entendu le mod\`ele du disque de Poincar\'e, la sph\`ere $\mathbb{S}^n$ est vue comme l'infini conforme de l'espace hyperbolique $\mathbb{H}^{n+1}$, avec $g_r = \frac{1}{4} (1-r^2)^2 g_{\mathbb{S}}$, o\`u $r = \frac{1-|x|}{1+|x|}$ et $g_{\mathbb{S}}$ est la m\'etrique canonique de $\mathbb{S}^n$. Pour les vari\'et\'es AHE qui poss\`ede une m\'etrique $g$ dans son infini conforme qui v\'erifie la condition d'Einstein $\ric^g=4\lambda(n-1)g$, alors on obtient\,:
\begin{equation}
 \label{g_+_einstein}
 g_r = (1-\lambda\,r^2)^2g.
\end{equation}
On appelle $(X,g_+)$ une vari\'et\'e de Poincar\'e--Einstein de $(M,[g])$, une vari\'et\'e AH d'infini conforme $(M,[g])$, qui v\'erifie que $g_+$ s'\'ecrive sous la forme (\ref{forme normale}) o\`u $g_r$ admet le m\^eme d\'eveloppement formel que dans le th\'eor\`eme \ref{theoreme_g+}. Ainsi une m\'etrique de Poincar\'e--Einstein v\'erifie une condition d'Einstein asymptotique\,:
\begin{align*}
 \ric^{g_+}\! +\, n\,g_+ &= O(r^{n-1}\log{r}) \mbox{, si\ } n \mbox{\ est pair,} \\
 \ric^{g_+}\! +\, n\,g_+ &= O(r^n) \mbox{, si\ }n \mbox{\ est impair.}
\end{align*}
   
\begin{remarque}
 Le fait de d\'eterminer la m\'etrique de Poincar\'e--Einstein de $(M,[g])$ est \'equivalent \`a un autre probl\`eme \`a bord; celui de d\'eterminer la m\'etrique ambiante de $(M,[g])$. On pourra consulter a ce sujet \cite{MR1112625} et \cite{MR837196}.
\end{remarque}
Donnons quelques exemples de m\'etrique de Poincar\'e en basse dimension, pour cela d\'efinissons quelques tenseurs classiques de g\'eom\'etrie riemannienne. Soit $(M,g)$ une vari\'et\'e riemannienne de dimension $n$, on appelle tenseur de Schouten de $g$, le tenseur $\Sc_n$ suivant\,:
\begin{equation}
 \label{tenseur-de-schouten}
  \Sc := \frac{1}{n-2}\ric - \frac{\scal}{2(n-1)(n-2)}\, g,
 \end{equation}
o\`u $\ric$ et $\scal$ se rapportent \`a $g$. On note $\W$ le tenseur de Weyl de $g$ et $(e_1,\ldots,e_n)$ une base orthonorm\'ee de $TM$ par rapport \`a $g$, on d\'efinit $\B$ le tenseur de Bach de $g$ de la mani\`ere suivante\,:
\begin{equation*}
 \B(X,Y) := \sum_{k=1}^n (\nabla_{\textstyle{e_k}} \nabla_{\textstyle{e_k}}\Sc) (X,Y) - (\nabla_{\textstyle{e_k}} \nabla_{\textstyle{Y}}\Sc) (X,e_k) - \Sc(\W_{\textstyle{e_k,X}} Y,e_k),
\end{equation*}
o\`u $X$ et $Y$ sont deux champs de vecteurs de $TM$. Quand $n=2$, la m\'etrique $g_r$ admet le d\'eveloppement asymptotique suivant\,:
\begin{equation}
 \label{met_poincare_2}
 g_r = g + g_{(2)} r^2 + O(r^3) \mbox{\ avec\ } \tr g_{(2)} = - \frac{1}{2}\scal^g,
\end{equation}
o\`u $\tr$ d\'esigne la trace par rapport \`a $g$. Pour $n=4$, on obtient \,:
\begin{equation}
 \label{met_poincare_4}
 g_r = g - \Sc r^2 - \frac{1}{3}\B r^4\log{r} + g_{(4)}\,r^4 + O(r^5) \mbox{\ avec\ } \tr g_{(4)} = \frac{1}{4}(\tr \Sc\!\circ\!\Sc),
\end{equation}
on retrouve le fait que le tenseur de Bach est covariant conforme en dimension $4$. Si $n=6$, alors la m\'etrique $g_r$ s'\'ecrit dans un voisinage du bord\,:
 \begin{equation}
  \label{met_poincare_6}
  g_r = g - \Sc r^2 + \big(\frac{1}{4} \Sc\circ\Sc - \frac{1}{8} \B\big) \,r^4 + h\,r^6\log{r} + g_{(6)}\,r^6 + O(r^7).
 \end{equation}
Pour plus de d\'etails sur la m\'etrique de Poincar\'e et les applications qui en d\'ecoulent, on pourra consulter l'excellent livre \cite{MR2650244}.

%%%%%%%%%%%%%%%%%%%%%%%%%%%%%%%%%%%%%%%%%%%%%%%%%%%%%%%%%%%%%%%%%%%%%%%%%%%%%%%%%%%%%%%%%%%%%%%%%%%%%%%%%%%%%%%%
%%%%%%%%%%%%%%%%%%%%%%%%%%%%%%%%%%%%%%%%%%%%%%%%%%%%%%%%%%%%%%%%%%%%%%%%%%%%%%%%%%%%%%%%%%%%%%%%%%%%%%%%%%%%%%%%
\section{Les applications conforme--harmoniques}
%%%%%%%%%%%%%%%%%%%%%%%%%%%%%%%%%%%%%%%%%%%%%%%%%%%%%%%%%%%%%%%%%%%%%%%%%%%%%%%%%%%%%%%%%%%%%%%%%%%%%%%%%%%%%%%%
%%%%%%%%%%%%%%%%%%%%%%%%%%%%%%%%%%%%%%%%%%%%%%%%%%%%%%%%%%%%%%%%%%%%%%%%%%%%%%%%%%%%%%%%%%%%%%%%%%%%%%%%%%%%%%%%

On munit notre vari\'et\'e compacte $M^n$ d'une structure conforme $[g]$ et on note $g_+=r^{-2}(dr^2+g_r)$ sa m\'etrique de Poincar\'e d\'efinie sur $X = M \times ]0,\epsilon[$.  Il convient de remarquer que $g_+$ explose pour $r=0$, cependant on peut quand m\^eme d\'efinir le laplacien pour les applications de $(\overline{X},g_+)$ sur $(N,h)$. On se donne $\varphi$ une application $C^{\infty}$ de $M$ dans $N$, notre probl\`eme \`a bord est de d\'eterminer $\tilde{\varphi}$ une application $C^{\infty}$ de $\overline{X}$ dans $N$ qui soit solution du syst\`eme suivant\,:
\begin{center}
 \label{systeme}
 $\left\{\begin{array}{ll}
  \tilde{\varphi}_{|r=0} & = \varphi \\ \delta^{g_+}T{\tilde{\varphi}} & = 0.
  \end{array}
 \right.$
\end{center}
Notons $p_M$ la projection de $M \times [0,1]$ sur $M$, gr\^ace \`a l'exponentielle, on va identifier localement notre vari\'et\'e d'arriv\'ee $N$, avec le fibr\'e $(\varphi \circ p_M)^*TN$, de mani\`ere \`a faire un d\'eveloppement asymptotique sur ce fibr\'e. Quand la dimension de $M$ est paire, on obtient le th\'eor\`eme suivant\,:
\begin{thrm}
 \label{theoreme_phi_pair}
 Supposons que $n$ soit un entier pair, on se donne $(M^n,g)$ et $(N,h)$ deux vari\'et\'es riemanniennes et on note $(X,g_+)$ la m\'etrique de Poincar\'e de $(M,g)$. On \'ecrit $g_+$ sous la forme (\ref{forme normale}) et on se donne $\varphi$ une application $C^{\infty}$ de $(M,g)$ dans $(N,h)$, alors il existe une unique section $U$ de $(\varphi \circ p_M)^*TN$ modulo $O(r^n)$ d\'efinie dans un voisinage de $M$ dans $X$, telle que l'application $\tilde{\varphi} := (\exp_{\varphi\,\circ\,p_M}) \circ U$ soit solution du syst\`eme suivant\,:
 \begin{center}
  $\left\{\begin{array}{ll}
   \tilde{\varphi}_{|r=0} & = \varphi \\
   \delta^{g_+}T\tilde{\varphi} & = O(r^{n+1}\log{r}).
  \end{array}\right.$
 \end{center} 
 
 Plus pr\'ecis\'ement, $U$ admet le d\'eveloppement asymptotique en $r=0$ suivant\,:
 \begin{equation}
  U = U_2\, r^2 + \cdots + U_{n-2}\, r^{n-2} + H^g\, r^n \log{r} + U_n\, r^n + \ldots,
 \end{equation}
 o\`u les premiers points d\'esignent des termes en puissances de $r$ paires qui sont enti\`erement d\'etermin\'es par $\varphi$ et des termes de courbures de $g$ et de $h$. Le terme $H^g$ ne d\'epend que de $\varphi$ et de $[g]$ et l'\'equation $H^g(\varphi)=0$ est une \'equation aux d\'eriv\'ees partielles elliptique non--lin\'eaire d'ordre $n$ sur des applications de $(M^n,g)$ dans $(N,h)$, qui est covariante conforme par rapport \`a $g$.
 
 En outre, notre terme $H^g(\varphi)$ est de la forme suivante\,:
 \begin{equation*}
  H^g(\varphi) = a_n\, (\delta^g d)^{n/2-1} \delta^g T\varphi + \mbox{des d\'eriv\'ees de $\varphi$ d'ordre inf\'erieurs},
 \end{equation*}
 o\`u $a_n := \frac{(-1)^{n/2-1}}{2^{n-1} (n/2)! (n/2-1)!}$ et v\'erifie $H^{\overline{g}}(\varphi) = e^{-n\omega} H^g(\varphi)$ pour $\overline{g} = e^{2\omega} g$.
 %o\`u $a_n := (-1)^{n/2-1} [2^{n-1} (n/2)! (n/2-1)!]^{-1}$ et v\'erifie $H^{\overline{g}}(\varphi) = e^{-n\omega} H^g(\varphi)$ pour $\overline{g} = e^{2\omega} g$.
\end{thrm}

Nous avons le th\'eor\`eme suivant quand la dimension de $M$ est impaire\,:
\begin{thrm}
 \label{theoreme_phi_impair}
 Supposons que $n$ soit impair, on se donne $(M^n,g)$ et $(N,h)$ deux vari\'et\'es riemanniennes et on note $(X,g_+)$ la m\'etrique de Poincar\'e de $(M,g)$. On \'ecrit $g_+$ sous la forme (\ref{forme normale}) et on se donne $\varphi$ une application $C^{\infty}$ de $(M,g)$ dans $(N,h)$, alors il existe une unique section $U$ de $(\varphi \circ p_M)^*TN$ modulo $O(r^n)$ d\'efinie dans un voisinage de $M$ dans $X$, telle que l'application $\tilde{\varphi} := (\exp_{\varphi\,\circ\,p_M}) \circ U$ soit solution du syst\`eme suivant\,:
 \begin{center}
  $\left\{\begin{array}{ll}
   \tilde{\varphi}_{|r=0} & = \varphi \\
   \delta^{g_+}T\tilde{\varphi} & = O(r^{n+1}).
  \end{array}\right.$
 \end{center} 
 
 Plus pr\'ecis\'ement, $U$ admet le d\'eveloppement asymptotique en $r=0$ suivant\,:
 \begin{equation}
  U = U_2\, r^2 + \cdots + U_n\, r^n + U_{n+1}\, r^{n+1} + \ldots,
 \end{equation}
 o\`u les premiers points d\'esignent des termes en puissances de $r$ paires qui sont enti\`erement d\'etermin\'es par $\varphi$ et des termes de courbures de $g$ et de $h$. Le terme $U_n$ est ind\'etermin\'e.
\end{thrm}

Nous pouvons \`a pr\'esent d\'efinir les applications conforme--harmoniques en dimension paire.

\begin{definition}
 %Les solutions de l'\'equation aux d\'eriv\'ees partielles covariante conforme du th\'eor\`eme \ref{theoreme_phi_pair} sont appel\'ees les applications conforme--harmoniques de $(M^n,[g])$ dans $(N,h)$. Afin d'all\'eger le texte, on parlera alors d'applications C--harmoniques.
 Nous appelons applications conforme--harmoniques de $(M^n,[g])$ dans $(N,h)$, les solutions de l'\'equation aux d\'eriv\'ees partielles covariante conforme du th\'eor\`eme \ref{theoreme_phi_pair}. On parlera alors d'applications C--harmoniques, afin d'all\'eger le texte.
\end{definition}

\begin{remarque}
 Nous aurions pu nous contenter de d\'eterminer la valeur sur le bord des $(n-1)$ premi\`eres d\'eriv\'ees par rapport \`a $r$ de notre solution $\tilde{\varphi}$ et de notre terme $H$ quand $n$ est pair. Il est facile de voir que c'est \'equivalent \`a la donn\'ee du d\'eveloppement asymptotique de $U$, mais il nous semble plus naturel de proc\'eder comme nous avons fait, en particulier pour faire le lien avec le th\'eor\`eme de Graham--Zworski sur les fonctions (voir ci--dessous).
\end{remarque}

Un exemple simple d'applications C--harmoniques est de regarder quand notre vari\'et\'e d'arriv\'ee $N$ est \'egale \`a $\mathbb{R}^m$, cela revient \`a travailler avec les fonctions $C^{\infty}$ de $(X,g_+)$. On retrouve quand $n$ est pair, la construction de Graham et Zworski (\cite{MR1965361}) des op\'erateurs GJMS de Graham, Jenne, Mason et Sparling (\cite{MR1190438}) en calculant directement le d\'eveloppement asymptotique de $\tilde{\varphi}$. C'est pourquoi dans le cas g\'en\'eral, comme on ne peut pas faire de d\'eveloppement asymptotique sur $\tilde{\varphi}$, on identifie notre vari\'et\'e $N$ d'arriv\'ee avec le fibr\'e $\varphi^*TN$, pour pouvoir calculer le d\'eveloppement asymptotique de $U$. Sur les fonctions, notre th\'eor\`eme \ref{theoreme_phi_pair} devient\,:

\begin{thrm}[Graham--Zworski]
 Soit $f$ une fonction $C^{\infty}$ de $M$, alors il existe une unique fonction $\tilde{f} \mod{O(r^n)}$ de $\overline{X}$ v\'erifiant le syst\`eme suivant\,:
 \begin{center}
  \label{systeme_fonction}
  $\left\{\begin{array}{ll}
   \tilde{f}_{|r=0} & = f \\
   \Delta_{g_+} \tilde{f} & = O(r^{n+1}\log{r}).
  \end{array}\right.$
 \end{center}
 De plus, le d\'eveloppement asymptotique de $f$ est pair jusqu'au terme $n-1$ et il contient un terme en $r^n\log{r}$ qui ne d\'epend que de $f$ et de $[g]$. Ce terme logarithmique d\'efinit un op\'erateur diff\'erentiel covariant conforme sur les fonctions de $(M,g)$ qui a pour terme principal $\Delta^{n/2}_g$. Il s'agit de l'op\'erateur GJMS de rang maximal.
\end{thrm}

Du th\'eor\`eme pr\'ec\'edent et d'une formule due \`a Graham dans \cite{MR2366901} qui a \'et\'e retrouv\'e par Gover (voir th\'eor\`eme $1.2$ dans \cite{MR2244375}), on obtient le corollaire suivant\,:
\begin{corollaire}
 \label{corollaire_fonction_einstein}
 Soient $(M^n,g)$ une vari\'et\'e d'Einstein de dimension paire et $f$ une fonction de $M$, alors $f$ est C--harmonique sur $(M,[g])$ si et seulement si\,:
\begin{equation*} 
 \Big( \prod_{j=1}^{n/2} \big( \Delta^g - \frac{(n+2j-2)\, (n-2j)}{4\, n\, (n-1)}\,\scal \big) \Big) f = 0.
\end{equation*}
\end{corollaire}

%%%%%%%%%%%%%%%%%%%%%%%%%%%%%%%%%%%%%%%%%%%%%%%%%%%%%%%%%%%%%%%%%%%%%%%%%%%%%%%%%%%%%%%%%%%%%%%%%
\subsection{D\'emonstration du th\'eor\`eme \ref{theoreme_phi_pair}}
%%%%%%%%%%%%%%%%%%%%%%%%%%%%%%%%%%%%%%%%%%%%%%%%%%%%%%%%%%%%%%%%%%%%%%%%%%%%%%%%%%%%%%%%%%%%%%%%%

%%%%%%%%%%%%%%%%%%%%%%%%%%%%%%%%%%%%%%%%%%%%%
\subsubsection{Quand $M$ est de dimension paire}
%%%%%%%%%%%%%%%%%%%%%%%%%%%%%%%%%%%%%%%%%%%%%

Soit $\tilde{\varphi}$ une application de $(\overline{X},g_+)$ dans $(N,h)$, on note $\varphi$ sa restriction sur $M$, on va montrer que si $\delta^{g_+} T\tilde{\varphi} = O(r^{n+1}\log{r})$, alors l'application $U := (\exp_{\varphi \circ p_M})^{-1} \circ \tilde{\varphi}$ admet le d\'eveloppement asymptotique annonc\'e.
 
Par changement conforme de m\'etrique, on obtient pour le laplacien de $\tilde{\varphi}$\,:
\begin{equation}
 \label{laplacien_g+}
 \delta^{g_+}T\tilde{\varphi}
  = r^2 \big(\delta^{g_r}T\tilde{\varphi} - \frac{\tr^{g_r}\!g'_r}{2}\,\partial_r\tilde{\varphi} - \nabla^h_{\!\textstyle{\partial_r\tilde{\varphi}}} \partial_r\tilde{\varphi} \big) + r(n-1)\,\partial_r\tilde{\varphi},
\end{equation}
et on obtient directement que $\delta^{g_+}T\tilde{\varphi}=O(r)$ par rapport \`a la m\'etrique $g$.
Pour simplifier les notations, on pose $\varphi^{(k)}$ comme \'etant \'egale \`a la valeur au bord de la $k$--i\`eme d\'eriv\'ee de $\tilde{\varphi}$ par rapport \`a $r$, c'est--\`a--dire
\begin{equation*}
 \varphi^{(k)} := \Big[(\nabla^h_{\!\textstyle{\partial_r\tilde{\varphi}}})^{k-1} \partial_r\tilde{\varphi}\Big]_{r=0}.
\end{equation*}
On a facilement les \'equivalences suivantes 
\begin{equation*}
 \delta^{g_+}T\tilde{\varphi}=O(r^2)
 \Longleftrightarrow [\nabla^h_{\!\textstyle{\partial_r\tilde{\varphi}}} \delta^{g_+}T\tilde{\varphi}]_{r=0}=0 
 \Longleftrightarrow \varphi^{(1)}=0.
\end{equation*}
Comme la d\'erivation $\nabla^h_{\!\textstyle{\partial_r\tilde{\varphi}}}$ sur $\tilde{\varphi}^*TN$ restreinte au bord ne d\'epend que de la m\'etrique $h$, de l'application $\varphi$ et de $\varphi^{(1)}$ qui est nul, on peut d\'eterminer $\varphi^{(k)}$ en fonction des conditions initiales, c'est--\`a--dire notre application $\varphi$ et des termes de courbures de $(M,g)$ et de $(N,h)$. On proc\`ede par r\'ecurrence sur $k$ tant que $k$ est strictement plus petit que $n$. Supposons que $\varphi^{(k-1)}$ soit d\'etermin\'e par les conditions initiales, on d\'etermine $\varphi^{(k)}$ en r\'esolvant l'\'equation $\delta^{g_+}T\tilde{\varphi}=O(r^{k+1})$ qui est \'equivalente \`a $\big[(\nabla^h_{\!\textstyle{\partial_r}\tilde{\varphi}})^k \delta^{g_+}T\tilde{\varphi} \big]_{r=0}=0$, c'est--\`a--dire\,:
\begin{equation}
 \label{equation_rec_derives_phi}
 (k-n)\,\varphi^{(k)} =(k-1) \Big[(\nabla^h_{\!\textstyle{\partial_r\tilde{\varphi}}})^{k-2}(\delta^{g_r}T\tilde{\varphi} - \frac{\tr^{g_r}\!g'_r}{2}\,\partial_r\tilde{\varphi}) \Big]_{r=0}.
\end{equation}
Comme on conna\^it les d\'eriv\'ees d'ordre inf\'erieur de $\tilde{\varphi}$ par hypoth\`ese de r\'ecurrence et le d\'eveloppement asymptotique de $g_r$ pour $r=0$ jusqu'au terme en $r^n\log{r}$, alors le terme de droite de (\ref{equation_rec_derives_phi}) est enti\`erement explicit\'e par les conditions initiales, tant que $k$ est strictement inf\'erieur \`a $n$. 

Par exemple si $k=2$ et $n\neq2$, on obtient\,:
\begin{equation}
 \label{phi''}
 \varphi^{(2)} = \frac{1}{2-n}\,\delta^g T\varphi,
\end{equation}
car la d\'eriv\'ee de $g_r$ par rapport \`a $r$ s'annule pour $r=0$ (voir th\'eor\`eme \ref{theoreme_g+}). Si $k$ est impair, on montre facilement par r\'ecurrence, en utilisant le fait que $g_r$ admet un d\'eveloppement asymptotique pair en $r=0$ jusqu'au terme $n-1$, que le terme de droite de (\ref{equation_rec_derives_phi}) est nul, ainsi pour tout entier impair $s$ compris entre $1$ et $n-1$, on a\,:
\begin{equation}
 \label{equation_rec_derives_impaires_phi}
 \varphi^{(s)} = 0.
\end{equation}
On verra que pour $n$ strictement plus grand que $2$, il appara\^it des termes de courbures de $(M,g)$ et de $(N,h)$ d\`es le terme $\varphi^{(4)}$ (on pourra consulter les exemples explicites de la cinqui\`eme partie).
 
Nous allons faire maintenant notre identification entre notre section $U$ et notre application $\tilde{\varphi}$. Pour cela, on prend $p$ un point de $M$, l'application exponentielle en $\varphi(p)$ d\'etermine un isomorphisme entre une petite boule $B_{\varphi(p)}$ de $N$ centr\'ee en $\varphi(p)$ et un ouvert de $T_{\varphi(p)}N$. On pose $\varepsilon_p := \sup(\{\alpha\ |\ \forall\beta<\alpha,\ \tilde{\varphi}(p,\beta)\in B_{\varphi(p)}\})$ et $U(p,r) := (\exp_{\varphi(p)})^{-1} \big(\tilde{\varphi}(p,r)\big)$, pour $r<\varepsilon_p$. On a facilement que $U(p,0) = (\exp_{\varphi(p)})^{-1} \big(\varphi(p)\big) = 0$. Comme la d\'eriv\'ee de $\tilde{\varphi}$ par rapport \`a $r$ est nulle sur le bord, il en est de m\^eme pour la d\'eriv\'ee de $U$ par rapport \`a $r$. On montre ainsi par r\'ecurrence, que les d\'eriv\'ees impaires d'ordre inf\'erieur \`a $n$ de $U$ s'annulent sur le bord. Ainsi les termes impairs du d\'eveloppement asymptotique de $U$ sont nuls jusqu'\`a l'ordre $n$ et les termes pairs sont donn\'es jusqu'\`a l'ordre $n-2$ par les d\'eriv\'ees de $\tilde{\varphi}$ par rapport \`a $r$ en $r=0$, qui sont eux--m\^emes enti\`erement d\'etermin\'es par les conditions initiales et des d\'eriv\'es de $(\exp_{\varphi(p)})^{-1}$ en $0$ qui sont elles--m\^emes des expressions universelles de $\R^h$ et de ses d\'eriv\'ees. En r\'esum\'e, le d\'eveloppement asymptotique de $U$ en $r=0$ est d\'ej\`a de la forme suivante\,:
\begin{equation*}
 U = U_2\, r^2 + \cdots + U_{n-2}\, r^{n-2} + \cdots.
\end{equation*}
Supposons que le terme suivant du d\'eveloppement asymptotique soit le terme $U_n\, r^n$, alors $\varphi^{(n)}$ existe et (\ref{equation_rec_derives_phi}) implique que
\begin{equation}
 \Big[(\nabla^h_{\!\textstyle{\partial_r\tilde{\varphi}}})^{n-2}(\delta^{g_r}T\tilde{\varphi} - \frac{\tr^{g_r}\!g'_r}{2}\,\partial_r\tilde{\varphi}) \Big]_{r=0} = 0.
\end{equation}
Cette \'equation n'a aucune chance d'\^etre vraie en g\'en\'eral, c'est pourquoi on introduit notre terme en $r^n\log{r}$. Le d\'eveloppement asymptotique en $r=0$ de la d\'eriv\'ee de $\tilde{\varphi}$ est ainsi la forme\,:
\begin{equation*}
 \partial_r\tilde{\varphi} = \tilde{\varphi}^{(2)}\, r + \frac{1}{3!}\, \tilde{\varphi}^{(4)}\, r^3 + \ldots + n\, \tilde{H}^g(\varphi)\, r^{n-1}\log{r} + \tilde{Q}\, r^{n-1} + \ldots,
\end{equation*}
o\`u $\tilde{\varphi}^{(2)}, \tilde{\varphi}^{(4)}, \ldots, \tilde{H}^g(\varphi)$ et $\tilde{Q}$ d\'esignent le transport parall\`ele le long de $r \rightarrow \tilde{\varphi}(r,\cdot)$ de $\varphi^{(2)}, \varphi^{(4)}, \ldots, H^g(\varphi)$ et $Q$. On obtient dans ce cas l\`a
\begin{equation*}
  r^2 \big(\nabla^h_{\!\textstyle{\partial_r\tilde{\varphi}}} \partial_r\tilde{\varphi} \big) - r\,(n-1)\,\partial_r\tilde{\varphi} = n\, H^g(\varphi)\, r^n +O(r^{n+1}\log{r}),
\end{equation*}
ce qui montre qu'avec (\ref{laplacien_g+}), l'\'equation $\delta^{g_+}T\tilde{\varphi}=O(r^{n+1}\log{r})$ est \'equivalente \`a l'\'egalit\'e suivante\,:
\begin{equation}
 \label{equation_H^g}
 H^g(\varphi) = \frac{n-1}{n!}\, \Big[(\nabla^h_{\!\textstyle{\partial_r\tilde{\varphi}}})^{n-2}(\delta^{g_r} T\tilde{\varphi} - \frac{\tr^{g_r}\!g'_r}{2}\,\partial_r\tilde{\varphi})\Big]_{r=0},
\end{equation}
ce qui d\'etermine $H^g(\varphi)$ par les conditions initiales. Notre \'equation de r\'ecurrence ne nous permet pas d'expliciter le terme $U_n$, il est formellement ind\'etermin\'e et l'unicit\'e de notre solution est donc bien v\'erifi\'ee modulo $O(r^n)$. L'existence se montre en remarquant que l'application $(p,r) \rightarrow \exp_{\varphi(p)} (U_2\, r^2 + \cdots + U_{n-2}\, r^n + H^g\, r^n\,\log{r})$ v\'erifie par construction le syst\`eme du th\'eor\`eme \ref{theoreme_phi_pair}.

%L'existence d'une telle solution se montre en remarquant que les termes $U_2,\ldots,U_{n-2}$ et $H^g$ construits pr\'ec\'edemment donne naissance \`a une application sont d\'etermin\'es Un tel d\'eveloppement asymptotique existe toujours, il suffit de poser $U_0 = 0$ et de d\'eterminer $U_1,\ldots,U_{n-1}$ par r\'ecurrence en r\'esolvant l'\'equation $\big[(\nabla^h_{\!\textstyle{\partial_r\tilde{\varphi}}})^k \big(\exp_{\varphi(p)} (\sum_{i=0}^k U_i\, r^i) \big) \big]_{r=0} = \varphi^{(k)}$ pour $k$ de $1$ \`a $n-1$. On v\'erifie alors que l'application $(p,r) \rightarrow \exp_{\varphi(p)} (U_2\, r^2 + \cdots + U_{n-2}\, r^n + H^g\, r^n\,\log{r})$ est une solution de notre th\'eor\`eme \ref{theoreme_phi_pair}.

Par un raisonnement classique sur les d\'eveloppements asymptotiques de ce type, on montre que $H^g$ est un terme covariant conforme (voir \cite{MR1758076}). Pour le calcul de la partie principale de $H^g$ voir la preuve de th\'eor\`eme \ref{theoreme_fonctionnelle_paire}.

%%%%%%%%%%%%%%%%%%%%%%%%%%%%%%%%%%%%%%%%%%%%%%%%%%%%%%%%%%%%%%%%%%%%%%%%%%%%%%%%%%%%%%%%%%%%%%%%%
\subsubsection{Quand $M$ est de dimension impaire}
%%%%%%%%%%%%%%%%%%%%%%%%%%%%%%%%%%%%%%%%%%%%%%%%%%%%%%%%%%%%%%%%%%%%%%%%%%%%%%%%%%%%%%%%%%%%%%%%%
 
Supposons maintenant que $n$ est impair, on a encore\,:
\begin{equation}
 \delta^{g_+}T\tilde{\varphi}
  = r^2 \big(\delta^{g_r}T\tilde{\varphi} - \frac{\tr^{g_r}\!g'_r}{2}\,\partial_r\tilde{\varphi} - \nabla^h_{\!\textstyle{\partial_r\tilde{\varphi}}} \partial_r\tilde{\varphi} \big) + r(n-1)\,\partial_r\tilde{\varphi}.
\end{equation}
On montre comme avant, que les d\'eriv\'ees impaires de $\tilde{\varphi}$ par rapport \`a $r$ d'ordre inf\'erieur \`a $n-1$ s'annulent sur le bord. Par contre, pour des raisons de parit\'e, le terme de droite de l'\'egalit\'e ci--dessus ne contient pas de terme en $r^n$ et il n'y a donc pas de terme en $r^n\log{r}$ dans le d\'eveloppement asymptotique de $U$ contrairement au cas pr\'ec\'edent. Le terme en $r^n$ est ind\'etermin\'e comme pr\'ec\'edemment.

On vient de montrer que si $\tilde{\varphi}$ est une application $C^{n-1}$ de $\overline{X}$ dans $N$ et qui est asymptotiquement harmonique de $(X,g_+)$ dans $(N,h)$, alors ce qu'on pourrait appeler son d\'eveloppement asymptotique (en fait celui de $U$) est d\'etermin\'e jusqu'au terme $r^{n-1}$ par les m\'etriques $g$ et $h$, et la valeur de $\tilde{\varphi}$ sur le bord.

%%%%%%%%%%%%%%%%%%%%%%%%%%%%%%%%%%%%%%%%%%%%%%%%%%%%%%%%%%%%%%
\subsection{Exemples}
%%%%%%%%%%%%%%%%%%%%%%%%%%%%%%%%%%%%%%%%%%%%%%%%%%%%%%%%%%%%%%

\begin{proposition}
\label{proposition_einstein_C-harmonique}
 Soient $(M^n,g)$ une vari\'et\'e d'Einstein de dimension paire et $(N,h)$ une vari\'et\'e riemannienne, alors les applications harmoniques de $(M,g)$ dans $(N,h)$ sont C--harmoniques.
\end{proposition}

\begin{proof}
 Supposons que $g$ v\'erifie $\ric=4\lambda(n-1)\,g$, alors d'apr\`es la formule (\ref{g_+_einstein}), la m\'etrique de Poincar\'e de $g$ s'\'ecrit $g_+ = r^{-2} \big(dr^2 + (1-\lambda\,r^2)^2g\big)$. Soit $\varphi$ une application harmonique de $(M,g)$ dans $(N,h)$, on obtient alors avec l'\'egalit\'e (\ref{equation_rec_derives_phi}) pour $k=1$\,:
 \begin{equation*}
  \varphi^{(2)} = \frac{-1}{n-2}\, \Big[ \frac{\delta T\tilde{\varphi}}{(1-\lambda r^2)^2} + \frac{2\,\lambda n r}{1-\lambda r^2} \Big]_{r=0} = 0.
 \end{equation*}
 Par r\'ecurrence, on montre ainsi que les d\'eriv\'ees paires de $\tilde{\varphi}$ sont nulles en $r=0$ et donc que $H^g(\varphi)=0$.
\end{proof}

\begin{remarque}
 Comme l'application identit\'e d'une vari\'et\'e riemannienne est harmonique, on vient donc de montrer que si la vari\'et\'e est Einstein, alors elle est C--harmonique. On montrera qu'il existe des hypoth\`eses plus faibles qu'\^etre Einstein pour que l'identit\'e soit harmonique (voir le corollaire \ref{corollaire_id_dim4} pour la dimension $4$ et le th\'eor\`eme \ref{theoreme-fonctionnelle-dimension6-id} pour la dimension $6$).
\end{remarque}

%%%%%%%%%%%%%%%%%%%%%%%%%%%%%%%%%%%%%%%%%%%%%%%%%%%%%%%%%%%%%%
\subsection{Obstruction au remplissage harmonique}
%%%%%%%%%%%%%%%%%%%%%%%%%%%%%%%%%%%%%%%%%%%%%%%%%%%%%%%%%%%%%%

Pour les vari\'et\'es AH, on obtient comme corollaire du th\'eor\`eme \ref{theoreme_phi_pair}\,:
\begin{corollaire}
 \label{corollaire_rigidite_harmonique_bord}
 Soient $(X^{n+1},g_+)$ une vari\'et\'e AH de dimension impaire, d'infini conforme $(M,[g])$ et $(N,h)$ une vari\'et\'e riemannienne, alors les applications qui sont de classe $C^n$ de $\overline{X}$ dans $N$ et harmonique de $(X,g_+)$ dans $(N,h)$, v\'erifient le fait que leurs restrictions \`a $M$ est C--harmonique de $(M,[g])$ dans $(N,h)$.
\end{corollaire}

\begin{proof}
 Soit $\tilde{\varphi}$ une application $C^n$ de $\overline{X}$ dans $N$ et harmonique de $(X,g_+)$ dans $(N,h)$, alors l'application $U := (\exp_{\tilde{\varphi}\,\circ\,p_M})^{-1} \circ \tilde{\varphi}$ admet, d'apr\`es le th\'eor\`eme \ref{theoreme_phi_pair}, le d\'eveloppement asymptotique en $r=0$ suivant\,:
 \begin{equation}
  U = U_2\, r^2+ \cdots + U_{n-2}\, r^{n-2} + H^g \, r^n \log{r} + O(r^n),
 \end{equation}
 or $\tilde{\varphi}$ est $C^n$ sur $\overline{X}$, donc $H^g=0$ et $\tilde{\varphi}_{|M}$ est bien C--harmonique.
\end{proof}

%%%%%%%%%%%%%%%%%%%%%%%%%%%%%%%%%%%%%%%%%%%%%%%%%%%%%%%%%%%%%%%%%%%%%%%%%%%%%%%%%%%%%%%%%%%%%%%%%%%%%%%%%%%%%%
%%%%%%%%%%%%%%%%%%%%%%%%%%%%%%%%%%%%%%%%%%%%%%%%%%%%%%%%%%%%%%%%%%%%%%%%%%%%%%%%%%%%%%%%%%%%%%%%%%%%%%%%%%%%%%
\section{L'\'energie renormalis\'ee}
%%%%%%%%%%%%%%%%%%%%%%%%%%%%%%%%%%%%%%%%%%%%%%%%%%%%%%%%%%%%%%%%%%%%%%%%%%%%%%%%%%%%%%%%%%%%%%%%%%%%%%%%%%%%%%
%%%%%%%%%%%%%%%%%%%%%%%%%%%%%%%%%%%%%%%%%%%%%%%%%%%%%%%%%%%%%%%%%%%%%%%%%%%%%%%%%%%%%%%%%%%%%%%%%%%%%%%%%%%%%%

%%%%%%%%%%%%%%%%%%%%%%%%%%%%%%%%%%%%%%%%%%%%%
\subsection{Quand $M$ est de dimension paire}
%%%%%%%%%%%%%%%%%%%%%%%%%%%%%%%%%%%%%%%%%%%%%

Soient $\varphi$ une application de $(M,g)$ dans $(N,h)$ et $\tilde{\varphi}$ la solution donn\'ee par le th\'eor\`eme \ref{theoreme_phi_pair} (les termes ind\'etermin\'es de $\tilde{\varphi}$ n'auront aucune incidence dans la suite), on note% avec les notations du th\'eor\`eme \ref{theoreme_g+}
\begin{equation*}
 E_{g_+}(\tilde{\varphi},\rho) := \frac{1}{2} \int_{M \times [\rho;\varepsilon]} |T\tilde{\varphi}|^2_{g_+,h} dvol_{g_+}
\end{equation*}
l'\'energie de $\tilde{\varphi}$ dans le ruban $M \times [\rho;\varepsilon]$ par rapport \`a $g_+$ et $h$, qui d\'epend donc de l'identification au bord via la m\'etrique $g$. En renormalisant cette \'energie, on obtient le th\'eor\`eme suivant\,:

\begin{thrm}
\label{theoreme_fonctionnelle_paire}
 Le d\'eveloppement asymptotique de $E(\tilde{\varphi},\rho)$ en $\rho=0$ est de la forme suivante\,:
 \begin{equation*}
  E_{g_+}(\tilde{\varphi},\rho) = E_{2-n}\,\rho^{2-n} + \cdots + E_{-2}\,\rho^{-2} + F\,\log{\frac{1}{\rho}} + O(1),
 \end{equation*}
 o\`u les points d\'esignent des termes en puissances paires de $\rho$ de $2-n$ \`a $-2$ qui sont enti\`erement d\'etermin\'es par $\varphi$ et des termes de courbures de $g$ et de $h$.

 Le terme $F$ ne d\'epend que de $\varphi$ et de $[g]$ et de $h$, on peut ainsi d\'efinir une fonctionnelle invariante conforme $\mathcal{E}_g (\varphi) := -(n\,a_n)^{-1} \,F(\varphi,g)$. Plus pr\'ecis\'ement, elle v\'erifie $\mathcal{E}_{\overline{g}}(\varphi) = \mathcal{E}_g(\varphi)$ pour tout $\overline{g}$ dans $[g]$ et elle s'\'ecrit
 \begin{equation}
  \mathcal{E}_g (\varphi) = \frac{1}{2} \int_M \big\langle (\delta^g d)^{n/2-2} \delta^g T\varphi , \delta^g T\varphi \big\rangle_g \,dvol_g + \ldots,
 \end{equation}
 o\`u les points de suspension d\'esignent des int\'egrales sur $M$ de termes en d\'eriv\'ees de $\varphi$ d'ordre inf\'erieur.
 
 De plus, le gradient de notre fonctionnelle $\mathcal{E}_g$ associ\'e \`a $g$ est \'egale \`a $\frac{1}{a_n}\, H^g$, c'est--\`a--dire que quelque soit $\dot{\varphi}\in\Gamma(\varphi^*TN)$, on a
 \begin{equation*}
  d_{\varphi} \mathcal{E}_g(\dot{\varphi}) = \frac{1}{a_n} \int_M \langle\dot{\varphi},H^g\rangle_h \,dvol_g.
 \end{equation*} 
\end{thrm}
 
Le d\'eveloppement asymptotique de l'\'energie poss\`ede la m\^eme structure que celui du volume d'une vari\'et\'e AH calcul\'e par Graham dans \cite{MR1758076}, ce qui donne ainsi des r\'esultat de m\^eme nature. Notre terme logarithmique $F$ (qui est notre fonctionnelle $\mathcal{E}$) et le terme logarithmique $L$ dans l'\'etude du volume sont ainsi des invariants conformes. De plus, Graham et Hirachi montrent dans \cite{MR2160867}, que la variation infinit\'esimale de $L$ ne d\'epend que du terme logarithmique $h$ du d\'eveloppement de la m\'etrique $g_r$, jouant ainsi le m\^eme r\^ole que $H$ par rapport \`a notre fonctionnelle.

\begin{remarque}
 Graham et Hirachi ont \'enonc\'e ce th\'eor\`eme en terme de $Q$--courbure et de tenseur d'obstruction. En effet, \`a un facteur multiplicatif, l'int\'egrale de la $Q$--courbure est \'egale \`a $L$ (voir Graham et Zworski dans \cite{MR1965361}) et le tenseur d'obstruction est \'egale \`a $h$. On pourra \'egalement consulter les travaux de Pierre Albin (\cite{MR2509323}) sur le sujet.
\end{remarque}

\begin{proof}
 On commence par montrer que l'\'energie admet un d\'eveloppement asymptotique de ce type, remarquons d\'ej\`a que\,:
 \begin{equation*}
  E_{g_+}(\tilde{\varphi},\rho) = \frac{1}{2} \,\int_{M \times [\rho;\varepsilon]} \frac{|T\tilde{\varphi}|^2_{dr^2+g_r}}{r^{n-1}} dr\,dvol_{g_r}.
 \end{equation*}
D'apr\`es les th\'eor\`emes \ref{theoreme_g+} et \ref{theoreme_phi_pair}, les d\'eriv\'ees impaires d'ordre inf\'erieur \`a $n$ de $g_r$ et de $\tilde{\varphi}$ s'annule pour $r=0$, ainsi on a le d\'eveloppement asymptotique suivant\,:
 \begin{equation*}
  |T\tilde{\varphi}|^2_{dr^2+g_r} dvol_{g_r} = \big( e_0 + e_2\,r^2 + \cdots + e_{n-2}\,r^{n-2} + O(r^n\log{r}) \big) \ dvol_g,
 \end{equation*}
 ce qui donne\ le r\'esultat annonc\'e et la formule suivante\,:
 \begin{equation}
  \label{fonctionnelle_dimension_n}
  \mathcal{E}_g(\varphi) = \frac{n-1}{2\, n!\, a_n} \int_M \partial_r^{n-2} \big[ |T\tilde{\varphi}|^2_{dr^2+g_r} \,dvol_{g_r} \big]_{r=0}.
 \end{equation}
 En effet, on a pour le terme de gauche\,:
 \begin{equation*}
  \mathcal{E}_g (\varphi) = - \frac{1}{n\,a_n} \,F = \frac{1}{2n\,a_n}\, \int_M e_{n-2}\, dvol_g,
 \end{equation*}
 et pour le terme sous l'int\'egrale \`a droite\,:
 \begin{equation*}
  \partial_r^{n-2} \big[ |T\tilde{\varphi}|^2_{dr^2+g_r} \,dvol_{g_r} \big]_{r=0} = (n-2)!\, e_{n-2}\, dvol_g.
 \end{equation*}
 La fonctionnelle $\mathcal{E}_g : \varphi \rightarrow F$ est bien d\'efinie, car $F$ d\'epend seulement des $(n-2)$ premiers termes des d\'eveloppements asymptotiques de $g_r$ et de $\tilde{\varphi}$, qui sont d\'etermin\'es par les conditions initiales. L'invariance conforme de notre fonctionnelle est un r\'esultat classique de l'\'etude de ce type de d\'eveloppement asymptotiques (voir \cite{MR1758076}) et d'apr\`es la formule (\ref{equation_rec_derives_phi}) on a $\varphi^{(k)} = -\frac{k-1}{n-k}\, \delta^g d \varphi^{(k-2)} + \cdots$, o\`u les $\cdots$ repr\'esentent des termes en d\'eriv\'ees de $\varphi$ d'ordre inf\'erieurs. D'apr\`es la formule (\ref{fonctionnelle_dimension_n}), on a facilement par r\'ecurrence\,:
 \begin{align*}
  \mathcal{E}_g(\varphi)
   &= \frac{n-1}{n!\, a_n} \int_M \big\langle d\varphi^{(n-2)} , T\varphi \big\rangle_g \,dvol_g + \ldots \\
   &= -\frac{(n-1)(n-3)}{2\, n!\, a_n} \int_M \big\langle d\delta^g d\varphi^{(n-4)} , T\varphi \big\rangle_g \,dvol_g + \ldots \\
   &= \int_M \big\langle (\delta^g d)^{n/2-2} \delta^g T\varphi , \delta^g T\varphi \big\rangle_g \,dvol_g + \ldots .
 \end{align*}
 
 On va montrer que le gradient de notre fonctionnelle $\mathcal{E}_g$ est un terme de bord dans une int\'egration par parties. Soit $(\varphi_t)_{t\in[0,1]}$ une famille \`a $1$--param\`etre d'applications $C^{\infty}$ de $M$ dans $N$ v\'erifiant le syst\`eme suivant\,:
 \begin{center}
  $\left\{
   \begin{array}{ll}
   \varphi_0 & = \varphi \\
   \big[\partial_t\varphi_t\big]_{t=0} & = \dot{\varphi},
   \end{array}
  \right.$
 \end{center}
 alors d'apr\`es le th\'eor\`eme \ref{theoreme_phi_pair}, pour tout $t$ dans $[0,1]$, il existe une application $\tilde{\varphi_t}$ de $M\times [0,\varepsilon]$ dans $N$ qui v\'erifie\,:
 \begin{center}
  $\left\{
   \begin{array}{ll}
   \tilde{\varphi}_{t|r=0} & = \varphi_t \\
   \delta^{g_+}T{\tilde{\varphi_t}} & = O(r^{n+1}\log{r}).
   \end{array}
  \right.$
 \end{center}
 On munit $M\times [0,\varepsilon]\times[0,1]$ de la m\'etrique $\gamma=g_+ + dt^2$ et on pose $\varPhi(p,r,t):=\tilde{\varphi_t}(p,r)$ qui est une application de $M\times [0,\varepsilon]\times[0,1]$ dans $N$. Son application tangente $T\varPhi$ est donc une section du fibr\'e $\Omega(M\times [0,\varepsilon])\otimes\varPhi^*TN$, sur lequel on d\'efinit la connexion $\nabla^{\gamma,h}$. Comme $|T\tilde{\varphi_t}|^2_{g_+,h} = |T\varPhi|^2_{\gamma,h} - |\partial_t\varPhi|^2_h$ et $\nabla^{\gamma,h}T\varPhi$ est sym\'etrique, on obtient
 \begin{align*}
  \partial_t E_{g_+}(\tilde{\varphi}_t,\rho)
   &= \frac{1}{2} \,\partial_t \Big( \int_{M\times[\rho,\varepsilon]} |T\tilde{\varphi_t}|^2_{g_+,h} dvol_{g_+} \Big) \\
   &=   \int_{M\times[\rho,\varepsilon]} \Big( \big\langle \nabla^{\gamma,h}_{\!\textstyle{\partial_t}} T\varPhi , T\varPhi\big\rangle_{\gamma,h}
       - \big\langle\nabla^h_{\textstyle{\partial_t\varPhi}}(\partial_t\varPhi) , \partial_t\varPhi\big\rangle_h \Big) \, dvol_{g_+} \\
   &=   \int_{M\times[\rho,\varepsilon]} \Big( \big\langle \nabla^h_{\textstyle{T\varPhi}} (\partial_t\varPhi) , T\varPhi \big\rangle_{\gamma,h}
       - \big\langle\nabla^h_{\textstyle{T\varPhi}}(\partial_t\varPhi) , T\varPhi \big\rangle_{dt^2,h} \Big) \, dvol_{g_+} \\
   &=   \int_{M\times[\rho,\varepsilon]} \big\langle\nabla^h_{\textstyle{T\varPhi}} (\partial_t\varPhi) , T\varPhi\big\rangle_{g_+,h} dvol_{g_+},
 \end{align*}
 ce qui donne pour $t=0$\,:
 \begin{equation*}
  \Big[\partial_t E_{g_+}(\tilde{\varphi}_t,\rho)\Big]_{t=0} = \int_{M\times[\rho,\varepsilon]} \big\langle\nabla^h_{\textstyle{T\tilde\varphi}} [\partial_t\tilde{\varphi}_t]_{t=0} , T\tilde{\varphi}\big\rangle_{g_+,h} dvol_{g_+}.
 \end{equation*}
 Apr\`es une int\'egration par parties, on obtient que la diff\'erentielle de $\mathcal{E}$ est \'egale au terme en $\log{\rho}$ de l'expression suivante multipli\'e par $(n\,a_n)^{-1}$\,:
 \begin{equation*}
  \int_{M\times[\rho,\varepsilon]} \big\langle[\partial_t\tilde{\varphi_t}]_{t=0},\delta^{g_+}T\tilde{\varphi}\big\rangle_h dvol_{g_+}
   - \int_M \rho^{-n+1}\,\big\langle[\partial_t\tilde{\varphi_t}]_{t=0} , \partial_{\rho}\tilde{\varphi}\big\rangle_h dvol_{g_\rho}.
 \end{equation*}
 Comme $\delta^{g_+}T\tilde{\varphi}=O(r^{n+1}\log{r})$, $[\partial_t\tilde{\varphi_t}]_{t=0}=O(1)$ et $dvol_{g_+}=O(r^{-n-1})$, il n'y a pas de $\log{\rho}$ dans la premi\`ere int\'egrale, le terme recherch\'e est donc celui en $\rho^{n-1}\,\log{\rho}$ de $\partial_{\rho}\tilde{\varphi}$, qui est exactement $n\,H^g$. Ainsi, on a bien
 \begin{equation*}
  d_{\varphi} \mathcal{E}_g(\dot{\varphi}) = \frac{1}{a_n} \int_M \langle\dot{\varphi},H^g\rangle_h \,dvol_g.
 \end{equation*} 
\end{proof}

\begin{remarque}
 On peut trouver dans la litt\'erature (voir \cite{MR1692148}, \cite{MR2301373} et les r\'ef\'erences cit\'ees) une autre g\'en\'eralisation des applications harmoniques qui est non--conforme. Ce sont les applications biharmoniques, qui sont d\'efinies comme \'etant les points critiques de la bi\'energie\,:
 \begin{equation*}
  E^2_g(\varphi) := \frac{1}{2} \int_M |\delta^g T\varphi|_h^2 \,dvol_g.
 \end{equation*}
 Quand $(M,g)$ est conform\'ement plate, les applications C--harmoniques sont biharmoniques pour le bon changement conforme de m\'etrique. 
\end{remarque}

%%%%%%%%%%%%%%%%%%%%%%%%%%%%%%%%%%%%%%%%%%%%
\subsection{Quand $M$ est de dimension impaire}
%%%%%%%%%%%%%%%%%%%%%%%%%%%%%%%%%%%%%%%%%%%%

Soit $(X^{n+1},g_+)$ une vari\'et\'e AHE, on reprend les notations du th\'eor\`eme \ref{theoreme_g+}. Notons $\tilde{\mathcal{H}}$ l'espace des applications $C^{\infty}$ de $\overline{X}$ dans $N$ qui v\'erifient la condition d'harmonicit\'e asymptotique du th\'eor\`eme \ref{theoreme_phi_impair}, on va renormaliser l'\'energie par rapport \`a $g_+$ et $h$, de ces applications sur la vari\'et\'e compacte \`a bord $X_{\rho}:=\{r \geq \rho \}$, quand $\rho$ est un r\'eel dans $]0,\varepsilon[$ qui tend vers $0$. On obtient alors le th\'eor\`eme suivant\,:
\begin{thrm}
\label{theoreme_fonctionnelle_impaire}
 Soit $\tilde{\varphi} \in \tilde{\mathcal{H}}$, alors le d\'eveloppement asymptotique de $E(\tilde{\varphi},\rho)$ en $\rho=0$ est de la forme suivante\,:
 \begin{equation*}
  E_{g_+}(\tilde{\varphi},\rho) = E_{2-n}\,\rho^{2-n} + \cdots + E_{-1}\,\rho^{-1} + C + o(1),
 \end{equation*}
 o\`u les points d\'esignent des termes en puissances impaires de $\rho$ de $2-n$ \`a $-1$ qui sont enti\`erement d\'etermin\'es par $\tilde{\varphi}$ et des termes de courbures de $g$ et de $h$.
 
 Le terme constant $C$ est un invariant conforme, c'est--\`a--dire que $C(\overline{g},\tilde{\varphi}) = C(g,\tilde{\varphi})$ pour tout $\overline{g}$ dans $[g]$. De plus, sa variation infinit\'esimale dans $\tilde{\mathcal{H}}$ ne d\'epend que du terme ind\'etermin\'e $U_n$ du d\'eveloppement asymptotique de $U$ dans le th\'eor\`eme \ref{theoreme_phi_impair}, et elle est donn\'ee par la formule suivante\,:
 \begin{equation*}
  d_{\tilde{\varphi}} C(Z) = - n \int_M \langle Z_0,U_n \rangle_h \,dvol_g,
 \end{equation*} 
 o\`u $Z\in\Gamma(\tilde{\varphi}^*TN)$ est une d\'eformation infinit\'esimale de $\tilde{\varphi}$ dans $\tilde{\mathcal{H}}$ et $Z_0$ est la restriction de $Z$ au bord.
\end{thrm}

\begin{remarque}
 Pour d\'efinir notre invariant conforme quand $n$ \'etait pair, on avait seulement besoin d'une partie du d\'eveloppement limit\'e de $g_r$ et de $\tilde{\varphi}$, qui d\'ependaient exclusivement de termes de courbure de $g$ et de la valeur de $\tilde{\varphi}$ sur $M$, ainsi on avait une fonctionnelle parfaitement d\'efinie sur les applications de $M$ dans $N$. Maintenant si $n$ est impair, on a besoin de toute la m\'etrique $g_r$ et de toute l'application $\tilde{\varphi}$ pour avoir notre invariant conforme, ce qui fait appara\^itre des termes qui sont ind\'ependants du bord rendant impossible la construction d'une fonctionnelle analogue \`a celle du th\'eor\`eme \ref{theoreme_fonctionnelle_paire}.
\end{remarque}

L'analogie entre l'\'etude du d\'eveloppement asymptotique de l'\'energie et celui du volume, d\'ecrite dans le paragraphe pr\'ec\'edent se poursuit dans le cas impair. Notre terme $C$ joue maintenant le r\^ole du volume renormalis\'e $V$ (le terme constant dans le d\'eveloppement asymptotique du volume), ces termes sont ind\'ependants du choix du repr\'esentant dans l'infini conforme de $g_+$. De plus, Anderson, pour $n=3$ (voir \cite{MR1825268}) et Albin dans le cas g\'en\'eral (voir \cite{MR2509323}), ont montr\'e que la variation infinit\'esimale de $V$ ne d\'epend que du terme ind\'etermin\'e $g^{(n)}$ du d\'eveloppement de la m\'etrique $g_r$, jouant ainsi le m\^eme r\^ole que $U_n$ par rapport \`a $C$, compl\'etant notre parall\`ele.

\begin{proof}
 Remarquons d\'ej\`a que\,:
 \begin{equation*}
  E_{g_+}(\tilde{\varphi},\rho) = E_{g_+}(\tilde{\varphi},\epsilon) + \frac{1}{2} \,\int_{M \times [\rho;\varepsilon]} |T\tilde{\varphi}|^2_{g_+,h} dvol_{g_+}
 \end{equation*}
 et qu'avec les th\'eor\`emes \ref{theoreme_g+} et \ref{theoreme_phi_impair}, on obtient\,:
 \begin{align*}
 |T\tilde{\varphi}|^2_{g_+,h} dvol_{g_+}
 &= r^{1-n} |T\tilde{\varphi}|^2_{dr^2+g_r,h} dvol_{g_r} \,dr \\
 &= \big( e_{(0)}\,r^{1-n} + e_{(2)}\,r^{3-n} + \cdots + e_{(n-1)} + e_{(n)}\,r + \dots \big) \ dvol_g \,dr,
 \end{align*}
 ainsi $E_{g_+}$ admet bien le d\'eveloppement asymptotique annonc\'e. La preuve de l'invariance conforme du terme constant est un r\'esultat classique (voir \cite{MR1758076}).

 Soient $(\tilde{\varphi}_t)_{t\in[0,1]}$ une famille \`a $1$--param\`etre de $\tilde{\mathcal{H}}$ et $Z := \big[\partial_t\tilde{\varphi}_t\big]_{t=0}$, on obtient, en proc\'edant comme avant, que la diff\'erentielle de $F$ en $\tilde{\varphi}$ dans la direction $Z$ est \'egale au terme constant de l'expression suivante\,:
 \begin{equation*}
  \int_{X_{\rho}} \langle Z,\delta^{g_+}T\tilde{\varphi} \rangle_h \,dvol_{g_+} - \int_M \rho^{-n+1}\, \langle Z , \partial_{\rho}\tilde{\varphi} \rangle_h \,dvol_{g_\rho}.
 \end{equation*}
 Comme $Z=O(1)$, $\delta^{g_+}T\tilde{\varphi}=O(r^{n+1})$ et $dvol_{g_+}=O(r^{-n-1})$, il n'y a pas de terme constant dans la premi\`ere int\'egrale.
 D'apr\`es les th\'eor\`emes \ref{theoreme_phi_impair} et \ref{theoreme_g+}, la d\'eformation $Z$ admet un d\'eveloppement pair jusqu'au rang $n$, le terme $\partial_{\rho}\tilde{\varphi}$ admet un d\'eveloppement impair jusqu'au rang $n-2$ et $dvol_{g_\rho}$ admet un d\'eveloppement pair jusqu'au rang $n$. Ainsi le terme recherch\'e provient donc du terme en $\rho^{n-1}$ de $\partial_{\rho}\tilde{\varphi}$, qui est exactement $n\,U_n$ ce qui donne bien\,:
 \begin{equation*}
  d_{\tilde{\varphi}} C(Z) = - n \int_M \langle Z_0,U_n \rangle_h \,dvol_g,
 \end{equation*}
\end{proof}

%%%%%%%%%%%%%%%%%%%%%%%%%%%%%%%%%%%%%%%%%%%%%%%%%%%%%%
\subsection{Exemple}
%%%%%%%%%%%%%%%%%%%%%%%%%%%%%%%%%%%%%%%%%%%%%%%%%%%%%%

La proposition suivante nous donne la valeur de notre fonctionnelle $\mathcal{E}$ pour des applications harmoniques d'une vari\'et\'e $(M,g)$ d'Einstein de dimension paire dans une vari\'et\'e riemannienne $(N,h)$ quelconque.
\begin{proposition}
 \label{proposition_fonctionnelle_einstein_harmonique}
 Soient $(M^n,g)$ une vari\'et\'e d'Einstein de dimension paire avec $\ric^g=4\lambda(n-1)\,g$, $(N,h)$ une vari\'et\'e riemannienne et $\varphi$ une application harmonique de $(M,g)$ dans $(N,h)$, alors notre fonctionnelle en $\varphi$ est \'egale \`a\,:
 \begin{equation*}
 \mathcal{E}_g(\varphi) = 2^{n-3} \lambda^{n/2-1} (n-2)! \int_M |T\varphi|^2_{g,h} \,dvol_g.
\end{equation*}
 Dans le cas particulier de l'identit\'e de $(M,g)$, on obtient ainsi\,:
 \begin{equation*}
  \mathcal{E}_g(id_M) = 2^{n-3} \lambda^{n/2-1} (n-2)!\,n \,vol_g(M),
 \end{equation*}
 o\`u $vol_g(M)$ d\'esigne le volume de $M$ par rapport \`a $g$.
\end{proposition}

\begin{proof}
 D'apr\`es les \'egalit\'es (\ref{fonctionnelle_dimension_n}) et (\ref{g_+_einstein}), on a
 \begin{align*}
  \mathcal{E}_g(\varphi)
   &= \frac{n-1}{2\,n!\,a_n} \int_M \Big[ \partial_r^{n-2} \big( |T\tilde{\varphi}|^2_{dr^2+g_r} \,dvol_{g_r} \big) \Big]_{r=0} \\
   &= \frac{n-1}{2\,n!\,a_n} \int_M \big[ \partial_r^{n-2} \big( (1-\lambda\,r^2)^n |T\tilde{\varphi}|^2_{dr^2+g_r} \big) \big]_{r=0} \,dvol_g,
 \end{align*}
 comme $\varphi$ est harmonique, on sait avec (\ref{proposition_einstein_C-harmonique}), que les d\'eriv\'ees de $\tilde{\varphi}$ par rapport \`a $r$ s'annulent quand $r=0$, ainsi
 \begin{equation*}
  \big[ \partial_r^{n-2} \big( (1-\lambda\,r^2)^n |T\tilde{\varphi}|^2_{dr^2+g_r} \big) \big]_{r=0} = (-\lambda)^{n/2-1} \,\frac{\big((n-2)!\big)^2}{\big((n/2-1)!\big)^2} \,|T\varphi|^2_{g,h},
 \end{equation*}
 ce qui donne bien pour la fonctionnelle en $\varphi$\,:
 \begin{equation*}
  \mathcal{E}_g(\varphi) = 2^{n-3} \lambda^{n/2-1} (n-2)! \int_M |T\varphi|^2_{g,h} \,dvol_g.
 \end{equation*}
\end{proof}

%%%%%%%%%%%%%%%%%%%%%%%%%%%%%%%%%%%%%%%%%%%%%%%%%%%%%%%%%%%%%%%%%%%%%%%%%%%%%%%%%%%%%%%%%%%%%%%%%%%%%%%%%%%%%
%%%%%%%%%%%%%%%%%%%%%%%%%%%%%%%%%%%%%%%%%%%%%%%%%%%%%%%%%%%%%%%%%%%%%%%%%%%%%%%%%%%%%%%%%%%%%%%%%%%%%%%%%%%%%
\section{Etude en basses dimensions.}
%%%%%%%%%%%%%%%%%%%%%%%%%%%%%%%%%%%%%%%%%%%%%%%%%%%%%%%%%%%%%%%%%%%%%%%%%%%%%%%%%%%%%%%%%%%%%%%%%%%%%%%%%%%%%
%%%%%%%%%%%%%%%%%%%%%%%%%%%%%%%%%%%%%%%%%%%%%%%%%%%%%%%%%%%%%%%%%%%%%%%%%%%%%%%%%%%%%%%%%%%%%%%%%%%%%%%%%%%%%

%%%%%%%%%%%%%%%%%%%%%%%%%%%%%%%
\subsection{La dimension $4$.}
%%%%%%%%%%%%%%%%%%%%%%%%%%%%%%%

%%%%%%%%%%%%%%%%%%%%%%%%%%%%%%%
\subsubsection{Ecriture explicite}
%%%%%%%%%%%%%%%%%%%%%%%%%%%%%%%

En dimension $4$, on peut expliciter facilement la fonctionnelle du th\'eor\`eme \ref{theoreme_fonctionnelle_paire} et l'\'equation de ses points critiques, c'est l'objet du th\'eor\`eme suivant\,:
\begin{thrm}
 \label{theoreme-fonctionnelle-dimension4}
 Soient $(M^4,g)$ et $(N,h)$ deux vari\'et\'es riemanniennes, la fonctionnelle invariante conforme du th\'eor\`eme \ref{theoreme_fonctionnelle_paire} s'\'ecrit alors\,:
 \begin{equation*}
  \mathcal{E}_g^4 (\varphi) = \frac{1}{2} \int_M \Big( |\delta T\varphi|^2_h + \frac{2}{3}\scal |T\varphi|^2_{g,h} - 2\,(\ric \otimes h)(T\varphi,T\varphi) \Big) dvol,
 \end{equation*}
 o\`u $\delta$ d\'esigne la divergence sur le fibr\'e $\Omega(M) \times \varphi^*TN$. L'\'equation de ses points critiques est\,:
 \begin{equation*}
  \delta d\delta T\varphi + \delta (\frac{2}{3}\scal - 2\ric) T\varphi - \Se(\delta T\varphi) = 0,
 \end{equation*}
 o\`u $\ric$, $\scal$ et $dvol$ se rapportent \`a $g$ et $\Se$ est l'endomorphisme de $\varphi^* TN$ d\'efini de la mani\`ere suivante\,:
 \begin{equation*}
  \Se(X) = \sum_{i=1}^4 \R^h_{\textstyle{X,T\varphi(e_i)}} T\varphi(e_i),
 \end{equation*}
 o\`u $(e_1,\ldots,e_4)$ est une base orthonorm\'ee de $TM$ par rapport \`a $g$ et $\R^h$ est le tenseur de courbure de $(N,h)$.
\end{thrm}

\begin{remarque}
 Quand on travaille avec des fonctions, l'\'equation des points critiques de $\mathcal{E}_g^4$ est tout simplement l'\'equation du noyau de l'op\'erateur de Paneitz $\Pa_4$\,:
 \begin{equation*}
  \Pa_4 = \Delta^2 + \delta (\frac{2}{3}\scal - 2\ric)\,d,
 \end{equation*}
 o\`u $\Delta$ d\'esigne le laplacien de $g$.
\end{remarque}

%%%%%%%%%%%%%%%%%%%%%%%%%
\subsubsection{Rigidit\'e}
%%%%%%%%%%%%%%%%%%%%%%%%%

\begin{proposition}
 \label{proposition_harmonique=C-harmonique_dim4}
  Soient $(M^4,g)$ une vari\'et\'e d'Einstein de courbure scalaire positive ou nulle et $(N,h)$ une vari\'et\'e riemannienne de courbure sectionnelle n\'egative ou nulle, on se donne une application $\varphi$ qui est C--harmonique de $(M,[g])$ dans $(N,h)$, alors
 \begin{enumerate}
  \item l'application $\varphi$ est totalement g\'eod\'esique,
  \item si la courbure scalaire de $(M,g)$ est strictement positive, alors l'application $\varphi$ est constante,
  \item si la courbure sectionnelle de $(N,h)$ est strictement n\'egative, alors l'application $\varphi$ est constante ou a une g\'eod\'esique comme image.
 \end{enumerate}
\end{proposition}

\begin{proof}
 On va d'abord montrer que sous les hypoth\`eses de la proposition, la notion de C--harmonicit\'e se confond avec celle d'harmonicit\'e. On sait d\'ej\`a que l'harmonicit\'e implique la C--harmonicit\'e d'apr\`es le corollaire \ref{proposition_einstein_C-harmonique}. Soit $\varphi$ une application C--harmonique de $(M,[g])$ dans $(N,h)$, la m\'etrique $g$ \'etant d'Einstein, on a\,:
 \begin{equation*}
  \delta d\delta T\varphi + \frac{1}{6}\scal\delta T\varphi - \Se(\delta T\varphi) = 0.
 \end{equation*}
 En prenant le produit scalaire de l'\'egalit\'e pr\'ec\'edente contre $\delta T\varphi$ par rapport \`a $g$ et $h$, on obtient avec une int\'egration par parties et en se souvenant que $\langle \Se(\delta T\varphi),\delta T\varphi \rangle_h \leq 0$\,:
 \begin{equation*}
  \int_M \Big( |d\delta T\varphi|^2_{g,h} + \frac{1}{6}\scal |\delta T\varphi|^2_h \Big) \,dvol \leq 0.
 \end{equation*}
 Si la courbure scalaire de $g$ est strictement positive, alors $\varphi$ est harmonique de $(M,g)$ dans $(N,h)$. Maintenant si la courbure scalaire de $g$ est nulle, alors $d\delta T\varphi=0$ et dans ces conditions, le produit scalaire de $d\delta T\varphi$ contre $T\varphi$ par rapport \`a $(g,h)$ donne avec une int\'egration par parties\,:
 \begin{equation*}
  0 = \int_M  \langle d\delta T\varphi,T\varphi \rangle_{g,h} dvol = \int_M |\delta T\varphi|^2_h dvol.
 \end{equation*}
 Ainsi $\varphi$ est encore harmonique de $(M,g)$ dans $(N,h)$. Il suffit alors d'appliquer un r\'esultat de rigidit\'e sur les applications harmoniques due \`a Eells et Sampson dans \cite{MR0164306} pour conclure.
\end{proof}

D'apr\`es le th\'eor\`eme \ref{proposition_einstein_C-harmonique}, on sait que l'identit\'e sur une vari\'et\'e d'Einstein de dimension paire est C--harmonique, l'objet du corollaire ci--dessous est de donner une condition plus faible sur les vari\'et\'es de dimension $4$ pour que l'identit\'e reste C--harmonique.
\begin{corollaire}
 \label{corollaire_id_dim4}
 Soit $(M^4,g)$ une vari\'et\'e riemannienne, l'application identit\'e de $M$ est C--harmonique de $(M,[g])$ dans $(M,g)$ si et seulement si la courbure scalaire de $g$ est constante.
\end{corollaire}

\begin{proof}
 D'apr\`es le th\'eor\`eme \ref{theoreme-fonctionnelle-dimension4}, l'identit\'e est C--harmonique si et seulement si  $\frac{2}{3}\delta\scal - 2\, \delta\ric= 0$, ce qui est \'equivalent \`a ce que la courbure scalaire soit constante.
\end{proof}

%%%%%%%%%%%%%%%%%%%%%%%%%%%%%%%%%%%%%%%%%%%%%%%%%%%%%%%%%%%%%%%%%%%%%%%%%%%%%%
\subsubsection{D\'emonstration du th\'eor\`eme \ref{theoreme-fonctionnelle-dimension4}}
%%%%%%%%%%%%%%%%%%%%%%%%%%%%%%%%%%%%%%%%%%%%%%%%%%%%%%%%%%%%%%%%%%%%%%%%%%%%%%
Soient $(M^4,g)$ une vari\'et\'e conforme et $g_+ = r^{-2}(dr^2+g_r)$ sa m\'etrique de Poincar\'e, d'apr\`es (\ref{met_poincare_4}), $g_r$ s'\'ecrit dans un voisinage du bord\,:
\begin{equation}
 \label{equation_g_r_4}
 g_r=g + \big(\frac{1}{12}\scal g-\frac{1}{2}\ric\big) r^2 + O(r^4\log{r}).
\end{equation}
On se donne $\varphi$ une application $C^{\infty}$ de $M$ dans $N$ et $\tilde{\varphi}$ l'application du th\'eor\`eme \ref{theoreme_phi_pair}, d'apr\`es (\ref{equation_rec_derives_phi}), (\ref{equation_rec_derives_impaires_phi}) et (\ref{equation_H^g}), on a $\varphi^{(1)} = 0$, $\varphi^{(2)} = -\frac{1}{2}\, \delta T\varphi$, $\varphi^{(3)} = 0$ et en notant $g''=[\partial^2_r g_r]_{r=0}$, on obtient pour le terme $H^g(\varphi)$\,:
\begin{align*}
 H^g(\varphi)
 &= \frac{1}{8} \Big[ (\nabla^h_{\!\textstyle{\partial_r\tilde{\varphi}}})^2 (\delta^{g_r}T\tilde{\varphi} - \frac{\tr^{g_r}\!g'_r}{2}\, \partial_r\tilde{\varphi}) \Big]_{r=0} \\
 &= \frac{1}{8} \Big( \Big[(\nabla^h_{\!\textstyle{\partial_r\tilde{\varphi}}})^2 (\delta T\tilde{\varphi})\Big]_{r=0} + \Big[ (\nabla_{\!\textstyle{\partial_r\tilde{\varphi}}})^2 (\delta^{g_r}T\varphi)\Big]_{r=0} - (\tr g'')\,\varphi^{(2)} \Big).
\end{align*}
Le premier terme se calcule avec (\ref{equation_P^2_phi}),
\begin{equation*}
 \Big[(\nabla^h_{\!\textstyle{\partial_r\tilde{\varphi}}})^2 (\delta T\tilde{\varphi})\Big]_{r=0}
    = (\delta d - \Se)\, \varphi^{(2)}
    = -\frac{1}{2}\,\delta d\delta T\varphi + \frac{1}{2}\,\Se(\delta T\varphi).
\end{equation*}
et les deuxi\`eme et troisi\`eme termes s'obtiennent avec (\ref{equation_P^2_g}) et (\ref{equation_g_r_4})\,:
\begin{equation*}
 \Big[(\nabla^h_{\!\textstyle{\partial_r\tilde{\varphi}}})^2 (\delta^{g_r}T\varphi)\Big]_{r=0} - (\tr g'')\,\varphi^{(2)} 
 = - \frac{1}{3}\,\delta(\scal T\varphi) + \delta(\ric T\varphi).
\end{equation*}
o\`u $\delta$ et la trace sont pris par rapport \`a $g$. On obtient ainsi\,:
\begin{equation*}
 H^g(\varphi) = - \frac{1}{16} \Big( \delta d\delta T\varphi - \Se(\delta T\varphi) + \delta\big(\frac{2}{3}\scal -2\ric\big) T\varphi \Big).
\end{equation*}
D'apr\`es (\ref{fonctionnelle_dimension_n}), notre fonctionnelle $\mathcal{E}_g^4$ est \'egale \`a\,:
\begin{align*}
 \mathcal{E}_g^4(\varphi)
  &= \frac{1}{16\,a_4} \int_M \partial^2_r \Big[|T\tilde{\varphi}|^2_{dr^2+g_r,h} dvol_{g_r}\Big]_{r=0} \\
  &= - \int_M \Big(\partial^2_r\Big[|T\tilde{\varphi}|^2_{dr^2+g_r,h}\Big]_{r=0} - \frac{\scal}{6}\, |T\varphi|^2_{g,h} \Big) dvol
\end{align*}
et on conclut en regardant le premier terme sous l'int\'egrale avec (\ref{equation_g_r_4})\,:
\begin{align*}
 \partial^2_r \Big[|T\tilde{\varphi}|^2_{dr^2+g_r,h}\Big]_{r=0}
 &= 2\,\big\langle \big[(\nabla^h_{\!\textstyle{\partial_r\tilde{\varphi}}})^2 T\tilde{\varphi}\big]_{r=0}, T\varphi\big\rangle_{g,h} - \,|T\varphi|^2_{g'',h} + 2\, |\varphi^{(2)}|^2_h\\
 &= - \big\langle d\delta T\varphi, T\varphi \big\rangle_{g,h} + \frac{1}{2}\, |\delta T\varphi|_h^2 - \frac{1}{6}\scal|T\varphi|^2_{g,h} \\
 &\ \ \ \ \  + \ric(T\varphi,T\varphi).
\end{align*}

%%%%%%%%%%%%%%%%%%%%%%%%%%%%%%%%%%%%%%%%%%
\subsection{La dimension $6$}
%%%%%%%%%%%%%%%%%%%%%%%%%%%%%%%%%%%%%%%%%%

%%%%%%%%%%%%%%%%%%%%%%%%%%%%%%%%%%%%%%%%%%
\subsubsection{\'Ecriture explicite}
%%%%%%%%%%%%%%%%%%%%%%%%%%%%%%%%%%%%%%%%%%

Obtenir des \'ecritures explicites de la condition de C--harmonicit\'e devient rapidement tr\`es compliqu\'e quand la dimension de $M$ augmente, mis \`a part le cas des fonctions d'une vari\'et\'e d'Einstein trait\'e dans le corollaire \ref{corollaire_fonction_einstein}, les calculs deviennent rapidement pharaoniques. Toutefois en supposant que la vari\'et\'e de d\'epart $M$ soit de dimension $6$ et que la vari\'et\'e d'arriv\'ee $N$ soit sym\'etrique, on a le r\'esultat suivant\,:

\begin{thrm}
 \label{theoreme-fonctionnelle-dimension6-einstein}
 Soient $(M,g)$ une vari\'et\'e d'Einstein de dimension $6$ avec $\ric^g=20\, \lambda\, g$ et $(N,h)$ une vari\'et\'e riemannienne sym\'etrique, on se donne une application $\varphi$ qui est $C^{\infty}$ de $M$ dans $N$. Alors $\varphi$ est C--harmonique de $(M,[g])$ dans $(N,h)$ si et seulement si
 \begin{equation}
  (\delta d - \Se + 16\,\lambda\,)(\delta d - \Se + 24\,\lambda)\,\delta T\varphi - 2\sum_{i=1}^6 \R^h_{\textstyle{\delta T\varphi,T\varphi(e_i)}} (\nabla^h_{\!\textstyle{T\varphi(e_i)}} \delta T\varphi) = 0.
 \end{equation}
 o\`u $\nabla^h$ est la connexion de $h$ sur $\varphi^*TN$, $\delta$ et $d$ se rapportent \`a $g$, et $\Se$ est d\'efini de mani\`ere analogue \`a la dimension $4$\,:
 \begin{equation}
  \Se(X) = \sum_{i=1}^6 \R^h_{\textstyle{X,T\varphi(e_i)}} T\varphi(e_i),
 \end{equation}
 o\`u $(e_1,\ldots,e_6)$ est une base orthonorm\'ee de $(M,g)$ et $R^h$ est le tenseur de courbure de $(N,h)$.
 
 De plus, notre fonctionnelle invariante conforme s'\'ecrit\,:
 \begin{equation*}
 \mathcal{E}_g^6(\varphi) = \frac{1}{2} \int_M \big( |d\delta T\varphi|^2_{g,h} - \langle \Se(\delta T\varphi),\delta T\varphi \rangle_h + 40\,\lambda\,|\delta T\varphi|^2_h + 384\,\lambda^2\,|T\varphi|^2_{g,h} \big) \,dvol.
 \end{equation*}
\end{thrm}

Toutefois, quand la vari\'et\'e $(M,g)$ est seulement riemannienne, on peut encore calculer la condition de C--harmonicit\'e et la valeur de notre fonctionnelle pour l'identit\'e de $(M,[g])$ dans $(M,g)$, c'est l'objet du th\'eor\`eme suivant\,:

\begin{thrm}
 \label{theoreme-fonctionnelle-dimension6-id}
 Soit $(M,g)$ une vari\'et\'e de dimension $6$, alors l'identit\'e est une application C--harmonique de $(M,[g])$ dans $(M,g)$ si et seulement si\,:
 \begin{equation*}
  (\Delta + \frac{5}{20}\scal g - \frac{7}{4} \ric)\, d\scal - \frac{5}{2}\tr(\nabla_{\textstyle{\ric}} \ric) + 20\, \delta\B + \frac{5}{4}\, d(|\ric|^2) = 0,
 \end{equation*}
 o\`u $\scal$, $\ric$ et $\B$ d\'esignent respectivement la courbure scalaire, le tenseur de Ricci et le tenseur de Bach de $g$.
 
 De plus, notre fonctionnelle en l'identit\'e est \'egale \`a\,:
 \begin{equation*}
  \mathcal{E}^6_g(id) = \frac{2}{25} \int_M \scal^2 \,dvol_g.
 \end{equation*}
\end{thrm}

%%%%%%%%%%%%%%%%%%%%%%%%%%%%%%%%%%%%%%%%%%%%%%%%%%%%%%%%%%%%%%%%%%%%%%%%%%%%%%%%%%%%%%%
\subsubsection{D\'emonstration du th\'eor\`eme \ref{theoreme-fonctionnelle-dimension6-einstein}}
%%%%%%%%%%%%%%%%%%%%%%%%%%%%%%%%%%%%%%%%%%%%%%%%%%%%%%%%%%%%%%%%%%%%%%%%%%%%%%%%%%%%%%%

Avec les \'egalit\'es (\ref{phi''}), (\ref{equation_rec_derives_phi}) et (\ref{equation_P^2_phi}), on obtient $\varphi^{(2)} = -\frac{1}{4}\,\delta T\varphi$ et $\varphi^{(4)} = \frac{3}{8}\,(\delta d - \Se + 8\,\lambda)\,\delta T\varphi$, ce qui donne avec (\ref{equation_H^g})\,:
\begin{equation}
 \label{H_dim6}
 144\, H^g(\varphi) = \big[(\nabla^h_{\!\textstyle{\partial_r\tilde{\varphi}}})^4 \delta T\tilde{\varphi}\big]_{r=0} + 12\, \lambda\, (\delta d - \Se + 12\,\lambda)\,\delta T\varphi.
\end{equation}
Soit $(e_1,\ldots,e_6)$ une base orthonorm\'ee de $(M,g)$, on obtient les deux \'egalit\'es suivantes en intervertissant les d\'eriv\'ees par rapport \`a $i$ et $r$\,:
\begin{align*}
 \big[(\nabla^h_{\!\textstyle{\partial_r\tilde{\varphi}}})^2 T\varphi(e_i)\big]_{r=0}
  &=  \nabla^h_{\!\textstyle{T\varphi(e_i)}} \varphi^{(2)} \\
 \big[(\nabla^h_{\!\textstyle{\partial_r\tilde{\varphi}}})^4 \big(T\tilde{\varphi}(e_i)\big)\big]_{r=0}
  &= \nabla^h_{\!\textstyle{T\varphi}(e_i)} \varphi^{(4)} + 3  \R^h_{\textstyle{\,\varphi^{(2)},T\tilde{\varphi}(e_i)}} \varphi^{(2)}.
\end{align*}
Int\'eressons nous au premier terme de (\ref{H_dim6})\,:
\begin{align*}
 (\nabla^h_{\!\textstyle{\partial_r\tilde{\varphi}}})^4 \delta T\tilde{\varphi} =
 & - \nabla^h_{\!\textstyle{T\tilde{\varphi}(e_i)}} (\nabla^h_{\!\textstyle{\partial_r\tilde{\varphi}}})^4 \big(T\tilde{\varphi}(e_i)\big) \\
 &  - (\nabla^h_{\!\textstyle{\partial_r\tilde{\varphi}}})^3 \big( \R^h_{\textstyle{\partial_r\tilde{\varphi},T\tilde{\varphi}(e_i)}} T\tilde{\varphi}(e_i) \big) \\
 &  - (\nabla^h_{\!\textstyle{\partial_r\tilde{\varphi}}})^2 \Big( \R^h_{\textstyle{\partial_r\tilde{\varphi},T\tilde{\varphi}(e_i)}} \big( \nabla^h_{\!\textstyle{\partial_r\tilde{\varphi}}} T\tilde{\varphi}(e_i) \big) \Big)  \\
 &  - \nabla^h_{\!\textstyle{\partial_r\tilde{\varphi}}} \Big(\R^h_{\textstyle{\partial_r\tilde{\varphi},T\tilde{\varphi}(e_i)}} \big((\nabla^h_{\!\textstyle{\partial_r\tilde{\varphi}}})^2 T\tilde{\varphi}(e_i)\big) \Big) \\
 &  - \R^h_{\textstyle{\partial_r\tilde{\varphi},T\tilde{\varphi}(e_i)}} \big((\nabla^h_{\!\textstyle{\partial_r\tilde{\varphi}}})^3 T\tilde{\varphi}(e_i)\big),
\end{align*}
comme $(N,h)$ est sym\'etrique et que les d\'eriv\'ees premi\`ere et troisi\`eme de $\tilde{\varphi}$ par rapport \`a $r$ s'annulent sur le bord, on a avec l'identit\'e de Bianchi\,:
\begin{equation*}
 \Big[(\nabla^h_{\!\textstyle{\partial_r\tilde{\varphi}}})^4 \delta T\tilde{\varphi}\Big]_{r=0} = (\delta d -\Se) \varphi^{(4)} - 12\R^h_{\textstyle{\,\varphi^{(2)},T\varphi(e_i)}} (\nabla^h_{\!\textstyle{T\varphi(e_i)}} \varphi^{(2)}).
\end{equation*}
D'apr\`es l'expression de $\varphi^{(2)}$ et $\varphi^{(4)}$, on obtient bien\,:
\begin{equation*}
  384\, H^g(\varphi) = (\delta d - \Se + 16\,\lambda\,)(\delta d - \Se + 24\,\lambda)\,\delta T\varphi - 2\R^h_{\textstyle{\delta T\varphi,T\varphi(e_i)}} (\nabla^h_{\!\textstyle{T\varphi(e_i)}} \delta T\varphi),
\end{equation*}
qui est la condition de C--harmonicit\'e \'enonc\'ee.

Nous allons calculer maintenant la fonctionnelle invariante conforme. La vari\'et\'e $(M,g)$ satisfait la condition d'Einstein, alors on a $g_r=(1-\lambda\,r^2)^2g$ et avec la formule (\ref{fonctionnelle_dimension_n}) et les notations du th\'eor\`eme \ref{theoreme_phi_pair}, on obtient
\begin{equation*}
 \mathcal{E}_g^6(\varphi) = \frac{4}{3} \,\int_M \Big[ \partial^4_r \big( (1-\lambda\,r^2)^6\, |T\tilde{\varphi}|^2_{dr^2+g_r,h} \big) \Big]_{r=0} \,dvol_g.
\end{equation*}
D'apr\`es les expressions de $\varphi^{(2)}$ et $\varphi^{(4)}$, on a\,:
\begin{align*}
 \big[ \partial^2_r (|T\tilde{\varphi}|^2_{g,h}) \big]_{r=0}
  &= -\frac{1}{2} \langle d\delta T\varphi,T\varphi \rangle_{g,h}, \\
 \big[ \partial^4_r (|T\tilde{\varphi}|^2_{g,h}) \big]_{r=0}
  &= \frac{3}{4} \langle d(\delta d - \Se + 8\lambda)\delta T\varphi,T\varphi \rangle_{g,h}\! - \frac{3}{8} \langle \Se(\delta T\varphi),\delta T\varphi \rangle_h\! + \frac{3}{8} |d\delta T\varphi|^2_{g,h},
\end{align*}
ce qui donne pour le terme sous l'int\'egrale\,:
\begin{align*}
 \lefteqn{\partial^4_r [ (1-\lambda\,r^2)^6\, |T\tilde{\varphi}|^2_{dr^2+g_r,h} ]_{r=0}} \\  
 = & \ \frac{3}{4} \langle d(\delta d - \Se + 8\lambda)\,\delta T\varphi,T\varphi \rangle_{g,h} - \frac{3}{8}\langle \Se(\delta T\varphi),\delta T\varphi \rangle_h \\
   & \ + \frac{3}{8} |d\delta T\varphi|^2_{g,h} + 24\lambda \langle d\delta T\varphi,T\varphi \rangle_{g,h} + 144\lambda^2\,|T\varphi|^2_{g,h} \\ 
   & \ - \frac{3}{4} \langle (\delta d - \Se + 8\lambda) \delta T\varphi,\delta T\varphi \rangle_h - 9\lambda|\delta T\varphi|^2_h.
\end{align*}
On conclut en int\'egrant par partie.

%%%%%%%%%%%%%%%%%%%%%%%%%%%%%%%%%%%%%%%%%%%%%%%%%%%%%%%%%%%%%%%%%%%%%%%%%%%%%%%%%%%%
\subsubsection{D\'emonstration du th\'eor\`eme \ref{theoreme-fonctionnelle-dimension6-id}}
%%%%%%%%%%%%%%%%%%%%%%%%%%%%%%%%%%%%%%%%%%%%%%%%%%%%%%%%%%%%%%%%%%%%%%%%%%%%%%%%%%%%

Posons $\varphi := Id_M$, alors $\varphi^{(2)}=0$ d'apr\`es (\ref{phi''}), $\varphi^{(4)} = \frac{3}{20}\,d\scal$ d'apr\`es (\ref{equation_rec_derives_phi}) et (\ref{equation_P^2_g}), et avec (\ref{equation_H^g}), on obtient\,:
\begin{equation*}
 144\,H = \Big[(\nabla_{\!\textstyle{\partial_r\tilde{\varphi}}})^{4}\delta T\tilde{\varphi}\Big]_{r=0} + \Big[(\nabla_{\!\textstyle{\partial_r\tilde{\varphi}}})^{4}\delta^{g_r} T\varphi\Big]_{r=0} - 2\, (\tr g'')\,\varphi^{(4)}.
\end{equation*}
Le premier terme s'exprime facilement en termes de courbures de $(M,g)$\,:
\begin{equation}
 \label{dim6-id-H-1er-terme}
 \Big[(\nabla_{\!\textstyle{\partial_r\tilde{\varphi}}})^{4}\delta^g T\tilde{\varphi}\Big]_{r=0} = \frac{3}{20}(\Delta - \ric)\,d\scal.
\end{equation}
Comme $g'''' = \frac{3}{2}\, g''\!\circ g'' - 3\B$, le deuxi\`eme terme s'\'ecrit avec (\ref{equation_P^2_g})\,:
\begin{equation*}
 \Big[(\nabla_{\!\textstyle{\partial_r\tilde{\varphi}}})^4 \delta^{g_r} T\varphi\Big]_{r=0} = \frac{3}{2}\,\delta(g''\!\circ g'') + 3\,\delta\B + \frac{3}{4}\,d(\tr g''\!\circ g''),
\end{equation*}
or $g''= -\frac{1}{2}\ric + \frac{1}{20}\scal g$ et $\delta(\ric\!\circ\!\ric) = -\frac{1}{2}\, d\scal\circ\ric - \tr(\nabla_{\textstyle{\ric}} \ric)$, ainsi
\begin{align}
 \label{dim6-id-H-2eme-terme}
 \Big[(\nabla_{\!\textstyle{\partial_r\tilde{\varphi}}})^4 \delta^{g_r} T\varphi\Big]_{r=0}
 & =  - (\frac{9}{400}\scal g + \frac{9}{80}\ric) (d\scal) - \frac{3}{8}\,\tr(\nabla_{\textstyle{\ric}} \ric)\notag\\
 & \ \ \ \ \   + 3\,\delta\B + \frac{3}{16}\, d|\ric|^2.
\end{align}
Finalement, avec (\ref{dim6-id-H-1er-terme}) et (\ref{dim6-id-H-2eme-terme}), on obtient\,:
\begin{equation*}
 144\, H = \frac{3}{20}(\Delta + \frac{5}{20}\scal\,g - \frac{7}{4} \ric)\,d\scal - \frac{3}{8}\tr(\nabla_{\textstyle{\ric}} \ric)+ 3\,\delta\B + \frac{3}{16}\, d(|\ric|^2),
\end{equation*}
ce qui donne l'\'equation de C--harmonicit\'e annonc\'ee.

D'apr\`es l'\'egalit\'e (\ref{fonctionnelle_dimension_n}), on obtient en dimension $6$\,:
\begin{equation*}
  \mathcal{E}^6_g(id) = \frac{4}{3} \int_M \big[ \partial_r^4 \big(|T\tilde{\varphi}|^2_{dr^2+g_r,g} \,dvol_{g_r} \big)\big]_{r=0}.
 \end{equation*}
Comme les d\'eriv\'ees premi\`eres et troisi\`emes de $g_r$ et $\tilde{\varphi}$ par rapport \`a $r$ s'annulent pour $r=0$, on obtient pour le terme sous l'int\'egrale\,:
\begin{equation*}
 \big[\partial_r^4 (|T\tilde{\varphi}|^2_{dr^2+g_r,g})\big]_{r=0} \,dvol_g + 6\, \big[ \partial_r^2 (|T\tilde{\varphi}|^2_{g_r,g}) \partial_r^2(dvol_{g_r}) \big]_{r=0} + 6\, \big[\partial_r^4 (dvol_{g_r})\big]_{r=0}.
\end{equation*}
Le dernier terme est \'egale \`a $\big(3 \tr g'''' - 9 \tr g''\!\circ g'' + \frac{9}{2}(\tr g'')^2\big)\, dvol_g$ et comme $\tr g'''' = \frac{3}{2} \tr g''\!\circ g''$, on a bien $ \mathcal{E}^6_g(id) = \frac{2}{25} \int_M \scal^2 \,dvol_g$.

%%%%%%%%%%%%%%%%%%%%%%%%%%%%%%%%%%%%%%%%%%%%%%%%%%%%%%%%%%%%%%%%%%%%%%%%%%%%%%%%%%%%%%%%%%%%%%%%%%%%%%%%%%%%%%%%%%%%%%%%%%%%%%%%%%%%%%%%%%%%%%%%%%%%%%%%%%%%%%%%%%%
%%%%%%%%%%%%%%%%%%%%%%%%%%%%%%%%%%%%%%%%%%%%%%%%%%%%%%%%%%%%%%%%%%%%%%%%%%%%%%%%%%%%%%%%%%%%%%%%%%%%%%%%%%%%%%%%%%%%%%%%%%%%%%%%%%%%%%%%%%%%%%%%%%%%%%%%%%%%%%%%%%%
\section{Un exemple d'application C--harmonique non--trivial}
%%%%%%%%%%%%%%%%%%%%%%%%%%%%%%%%%%%%%%%%%%%%%%%%%%%%%%%%%%%%%%%%%%%%%%%%%%%%%%%%%%%%%%%%%%%%%%%%%%%%%%%%%%%%%%%%%%%%%%%%%%%%%%%%%%%%%%%%%%%%%%%%%%%%%%%%%%%%%%%%%%%
%%%%%%%%%%%%%%%%%%%%%%%%%%%%%%%%%%%%%%%%%%%%%%%%%%%%%%%%%%%%%%%%%%%%%%%%%%%%%%%%%%%%%%%%%%%%%%%%%%%%%%%%%%%%%%%%%%%%%%%%%%%%%%%%%%%%%%%%%%%%%%%%%%%%%%%%%%%%%%%%%%%

%%%%%%%%%%%%%%%%%%%
\subsection{Situation}
%%%%%%%%%%%%%%%%%%%

Nous avons d\'efini une nouvelle famille d'applications entre deux vari\'et\'es riemanniennes, la question qui se pose ici est de comparer cette nouvelle notion d'harmonicit\'e avec celle, d\'ej\`a pr\'eexistante en dimension sup\'erieure.

Nous avons d\'ej\`a vu, avec la proposition \ref{proposition_einstein_C-harmonique}, que si la vari\'et\'e de d\'epart est une vari\'et\'e d'Einstein de dimension paire, alors les applications harmoniques sont C--harmoniques. D'autre part, il existe aussi des applications harmoniques qui ne sont pas C--harmoniques, il suffit de prendre l'identit\'e d'une vari\'et\'e riemannienne de dimension $4$ munie d'une m\'etrique \`a courbure scalaire non--constante (voir le corollaire \ref{corollaire_id_dim4}). Comme l'harmonicit\'e n'est pas une notion qui est covariante conforme en dimension plus grande que $2$, contrairement \`a la C--harmonicit\'e, il existe beaucoup d'applications C--harmoniques, qui ne soient pas harmoniques.

La question naturelle qui se pose, est donc l'existence d'une application C--harmonique, qui ne soit pas simplement non--harmonique pour une m\'etrique particuli\`ere, mais qui ne soit pas harmonique pour toute la classe conforme consid\'er\'ee pour la C--harmonicit\'e.

%%%%%%%%%%%%%%%%
\subsection{\'Enonc\'e}
%%%%%%%%%%%%%%%%

Notre strat\'egie est de partir d'une vari\'et\'e $(M,h)$ de dimension $4$ \`a courbure scalaire constante strictement n\'egative, on a vu que dans ce cas l\`a, l'identit\'e est une application harmonique et C--harmonique (corollaire \ref{corollaire_id_dim4}). On suppose que $h$ est proche d'une m\'etrique d'Einstein, on va montrer que fixer la m\'etrique $h$ dans la vari\'et\'e d'arriv\'ee et d\'eformer judicieusement la m\'etrique de la vari\'et\'e de d\'epart, permet de construire une application proche de l'identit\'e qui conserve la C--harmonicit\'e, mais qui n'est plus harmonique, pour n'importe quel changement conforme petit ou grand de m\'etrique, par rapport \`a la vari\'et\'e de d\'epart.

On note $\mathcal{M}^{k,\alpha}$ l'espace des m\'etriques riemanniennes $C^{k,\alpha}$ de $M$, $\mathcal{A}^{k,\alpha}$ l'espace des applications $C^{k,\alpha}$ de $M$ dans $M$ et $\Gamma^{k,\alpha}$ l'espace des sections $C^{k,\alpha}$ de $\varphi^*TM$, ce sont des espaces de Banach. 

On peut maintenant \'enoncer notre r\'esultat d'existence d'une application C--harmonique non--triviale\,:

\begin{thrm}
 \label{theoreme_non_triviale}
 Soit $(M,g_e)$ une vari\'et\'e d'Einstein de dimension $4$ \`a courbure scalaire strictement n\'egative, alors il existe $\epsilon>0$ tel que, pour toute m\'etrique lisse $h$ v\'erifiant $\|h-g_e\|_{k+4,\alpha}<\epsilon$, il existe $\varphi$ une application $C^{\infty}$ de $M$ dans $M$ et $g$ une m\'etrique $C^{\infty}$, telles que\,:
 \begin{enumerate}
  \item $\|g-g_e\|_{k+4,\alpha}<\epsilon$,
  \item $\varphi$ est C--harmonique de $(M,[g])$ dans $(M,h)$,
  \item quelque soit $\omega$ dans $C^{k+4,\alpha}$, $\varphi$ n'est pas harmonique de $(M,e^{2\omega}g)$ dans $(M,h)$.
 \end{enumerate} 
\end{thrm}
 
La d\'emonstration du th\'eor\`eme se fait en quatre \'etapes.
\begin{enumerate}
 \item On prouve gr\^ace au th\'eor\`eme des fonctions implicites, que pour n'importe quelle m\'etrique $g$ suffisamment proche de $h$, il existe une unique application $\varphi(g)$ qui soit \`a la fois proche de l'identit\'e et C--harmonique de $(M,[g])$ dans $(M,h)$ (lemme \ref{lemme_def_phi}).
 \item Gr\^ace encore au th\'eor\`eme des fonctions implicites, on montre que pour n'importe quelle m\'etrique $g$ suffisamment proche de $h$, il existe un unique changement conforme $\omega(g)$ qui soit \`a la fois petit et qui soit solution d'une \'equation plus faible que l'harmonicit\'e de $\varphi(g)$ de $(M,e^{2\omega}g)$ dans $(M,h)$ (lemme \ref{lemme_def_omega}).
 \item On construit une telle m\'etrique $g$ de fa\c{c}on \`a ce que l'application $\varphi(g)$ ne soit pas harmonique de $(M,e^{2\omega}g)$ dans $(M,h)$, pour de petits changements conforme $\omega$.
 \item On montre finalement avec le lemme \ref{lemme_inf}, que cette application $\varphi(g)$ n'est pas non plus harmonique de $(M,e^{2\omega}g)$ dans $(M,h)$, pour de grands changements conforme $\omega$.
\end{enumerate}

%%%%%%%%%%%%%%%%%%%%%%%%%%%%%%%%%%%%%%%%%%%%%%%%%%%%%%%%%%%%%%%%%%%%%%%%%%%%%%%%%%%
\subsection{D\'emonstration du th\'eor\`eme \ref{theoreme_non_triviale}}
%%%%%%%%%%%%%%%%%%%%%%%%%%%%%%%%%%%%%%%%%%%%%%%%%%%%%%%%%%%%%%%%%%%%%%%%%%%%%%%%%%%

%%%%%%%%%%%%%%%%%%%%%%%%%%%%%%%%%%%%%%%%%%%%%%%%%%%%%%%%%%%%%%%%%%%%%%%%%%%%%%%%%%%
\subsubsection{D\'eformation de l'\'equation de conforme--harmonicit\'e}
%%%%%%%%%%%%%%%%%%%%%%%%%%%%%%%%%%%%%%%%%%%%%%%%%%%%%%%%%%%%%%%%%%%%%%%%%%%%%%%%%%% 

Nous allons montrer comment obtenir des applications C--harmoniques qui sont proches de l'identit\'e, quand on se place suffisamment pr\`es d'une m\'etrique d'Einstein \`a courbure scalaire n\'egative.

Soient $g$ et $h$ deux m\'etriques dans $\mathcal{M}^{k+3,\alpha}$ et $\varphi$ dans $\mathcal{A}^{k+4,\alpha}$, on note\,:
\begin{equation*}
 \Pa^4(\varphi,g,h) := \delta d\delta T\varphi + \delta (\frac{2}{3}\scal^g - 2\ric^g) T\varphi - \Se^h(\delta T\varphi),
\end{equation*}
ainsi $\Pa^4(\varphi,g,h)=0$ est la condition de C--harmonicit\'e de $\varphi$ entre $(M,[g])$ et $(M,h)$ et $\Pa^4$ s'interpr\`ete comme une g\'en\'eralisation de l'op\'erateur de Paneitz aux applications de $(M,g)$ dans $(N,h)$.

On veut construire des applications C--harmoniques proches de l'identit\'e en faisant varier la m\'etrique $g$. Le lemme suivant nous donne l'existence d'un op\'erateur qui permet d'associer \`a chaque m\'etrique $g$ proche d'une m\'etrique $h$ particuli\`ere, l'unique application qui va \^etre \`a la fois proche de l'identit\'e et C--harmonique. Pour all\'eger les notations, on notera $\varphi$ cet op\'erateur.
  
\begin{lemme}
 \label{lemme_def_phi}
 On se donne $g_e$ une m\'etrique d'Einstein \`a courbure scalaire strictement n\'egative, alors le probl\`eme implicite 
 \begin{equation*}
  \Pa^4(\varphi(g),g,h)=0
 \end{equation*}
 admet une unique solution locale.

 Il existe un voisinage $V_e$ de $g_e$ dans $\mathcal{M}^{k+3,\alpha}$, un voisinage $V_{id}$ de $id_M$ dans $\mathcal{A}^{k+4,\alpha}$, un op\'erateur $\varphi$ continue qui de $V_e$ dans $V_{id}$ qui d\'epend continument de $h$, et tel que pour tout $(\psi,g,h) \in V_{id} \times V_e^2$, on a $\Pa^4(\varphi(g),g,h)=0$ si et seulement si $\psi = \varphi(g)$.
\end{lemme}

\begin{proof}   
 L'op\'erateur $\Pa^4$ est continue de $\mathcal{A}^{k+4,\alpha}(M) \times \big(\mathcal{M}^{k+3,\alpha}(M)\big)^2$ dans $\Gamma^{k,\alpha}(M)$ et s'annule en $(id,g_e,g_e)$. On obtient avec (\ref{equation_P^4_phi}) pour sa diff\'erentielle par rapport aux applications au point $(id,g_e,g_e)$ dans la direction $\dot{\varphi}$\,:
 \begin{equation}
  \label{dP4(phi)}
  \frac{\partial \Pa^4}{\partial \varphi}(\dot{\varphi}) = (\delta d - \frac{\scal^e}{12}) (\delta d -\frac{\scal^e}{4}) \,\dot{\varphi},
 \end{equation}
 o\`u $\scal^e$ est la courbure scalaire de $g_e$. Comme celle--ci elle strictement n\'egative, alors $\frac{\partial \Pa^4}{\partial \varphi}$ est inversible et on applique le th\'eor\`eme des fonctions implicites.
\end{proof}

%%%%%%%%%%%%%%%%%%%%%%%%%%%%%%%%%%%%%%%%%%%%%%% 
\subsubsection{Contr\^ole local du changement conforme}
%%%%%%%%%%%%%%%%%%%%%%%%%%%%%%%%%%%%%%%%%%%%%%%

Nous allons montrer qu'on contr\^ole le seul changement conforme local, qui puisse rendre harmonique notre application C--harmonique pr\'ec\'edemment construite.

Soient $\varphi$ une application $\mathcal{A}^{k+4,\alpha}$ et $\omega$ une fonction $C^{k+3,\alpha}$, on d\'esigne par $\Pa^2(\omega,\varphi,g,h)$ le laplacien de $\varphi$ de $(M,e^{2\omega}g)$ dans $(M,h)$, on obtient facilement avec l'\'egalit\'e de Bianchi\,:
\begin{equation*}
 \Pa^4(\varphi,g,h) = (\delta d + \frac{2}{3}\scal^g - \Se^h) \Pa^2(0,\varphi,g,h) + \frac{1}{3} \langle d\scal^g,T\varphi \rangle + 2 \langle \ric^g,\nabla T\varphi \rangle.
\end{equation*}
Pour contr\^oler localement le changement conforme, on utilise encore une fois le th\'eor\`eme des fonctions implicites, mais si on l'applique directement \`a $\Pa^4$, la courbure de $M$ va \^etre incompatible avec les hypoth\`eses du lemme \ref{lemme_def_phi}. On ajoute alors un terme correctif $Q$ \`a $\Pa^4$ qui va nous permettre d'utiliser le th\'eor\`eme des fonctions implicites avec les bonnes conditions de courbures. Soit $\overline{g}=e^{2\omega}g$, on pose\,:
\begin{align*}
 Q(\omega,\varphi,g,h) &:= \Pa^4(\varphi,\overline{g},h) - (\delta^{\overline{g}} d - \Se^h) \Pa^2(\omega,\varphi,g,h) \\
                       & =  \frac{2}{3} \scal^{\overline{g}} \Pa^2(\omega,\varphi,g,h) + \frac{1}{3} \langle d\scal^{\overline{g}},T\varphi \rangle_{\overline{g}} + 2 \langle \ric^{\overline{g}},\nabla^{\overline{g}} T\varphi \rangle_{\overline{g}},
\end{align*}
ainsi $Q$ est un op\'erateur de degr\'e $3$ en $\omega$, d'ordre $2$ en $\varphi$, d'ordre $3$ en $g$ et d'ordre $1$ en $h$. Le fait qu'une application $\varphi$ soit harmonique de $(M,\overline{g})$ dans $(M,h)$ et C--harmonique de $(M,[g])$ dans $(M,h)$ est donc \'equivalent au syst\`eme suivant\,:
\begin{center}
 $\left\{\begin{array}{ll}
  Q(\omega,\varphi,g,h)      &= 0, \\
  \Pa^2(\omega,\varphi,g,h)  &= 0.
 \end{array}\right.$
\end{center}

Soit $g$ et $h$ deux m\'etriques dans $V_e$, on va s'int\'eresser aux changements conformes $\omega$ qui v\'erifient la sous--condition d'harmonicit\'e suivante\,:
 \begin{equation}
  \label{systeme faible}
  \delta \Pa^2(\omega,\varphi(g),g,h) = 0,
 \end{equation}
o\`u $\delta$ d\'esigne la divergence par rapport \`a $g$. Le lemme suivant nous donne l'existence d'un op\'erateur qui permet d'associer pour chaque m\'etrique $g$ proche de $h$, l'unique changement conforme \`a une constante pr\`es, qui va \^etre \`a la fois proche de $0$ et solution de (\ref{systeme faible}). On peut remarquer que l'on aurait pu travailler avec la condition $\delta Q(\omega,\varphi(g),g,h) = 0$, qui est aussi naturelle et qui donne les m\^emes r\'esultats. On note $\omega$ cet op\'erateur, qui contr\^ole donc localement le changement conforme, dans le sens o\`u si l'application $\varphi(g)$ est harmonique pour un petit changement conforme de $g$, alors n\'ecessairement ce changement conforme sera \'egale \`a $\omega(g)$.
\begin{lemme}
 \label{lemme_def_omega}
 On se donne $g_e$ une m\'etrique d'Einstein \`a courbure scalaire strictement n\'egative, alors le probl\`eme implicite
 \begin{equation*}
  \delta \Pa^2(\omega(g),\varphi(g),g,h)=0
 \end{equation*}
 admet une unique solution locale.

 Il existe un voisinage $V'_e \subset V_e$ de $g_e$ dans $\mathcal{M}^{k+4,\alpha}$, un voisinage $V_0$ de la fonction nulle dans $C^{k+4,\alpha}$, un op\'erateur $\omega$ continue de $V'_e$ dans $V_0$ qui d\'epend continument de $h$, et tel que pour tout $(\Omega,g,h) \in V_0 \times (V'_e)^2$ avec $\int_M \Omega = 0$, on a $\delta \Pa^2\big(\Omega,\varphi(g),g,h)=0$ si et seulement si $\Omega = \omega(g)$.
\end{lemme}

\begin{proof}
 Comme $\delta \Pa^2\big(\omega + cste,\varphi,g,h) = 0 $ est \'equivalent \`a $ \delta P^2\big(\omega,\varphi,g,h) = 0$, le contr\^ole de $\omega$ se fait \`a une constante pr\`es, qu'on fixe en imposant que l'int\'egrale du changement conforme soit nulle sur $M$ par rapport \`a $g$. Au point $(0,id,g_e,g_e)$, l'op\'erateur $\delta \Pa^2$ s'annule et sa diff\'erentielle par rapport aux changements conformes qui sont d'int\'egrale nulle sur $M$ par rapport \`a $g_e$ est inversible. En effet, on obtient dans la direction $\dot{\omega}$ et au point $(0,id,g_e,g_e)$\,:
 \begin{equation*}
  \frac{\partial\delta \Pa^2}{\partial \omega} (\dot{\omega}) = - 2\, \Delta_e\dot{\omega}.
 \end{equation*}
 On conclut en appliquant le th\'eor\`eme des fonctions implicites.
\end{proof}

Le seul changement conforme local qui puisse rendre notre application $\varphi(g)$ harmonique est maintenant contr\^ol\'e par notre application $\omega$.

%%%%%%%%%%%%%%%%%%%%%%%%%%%%%%%%%%%%%%%%%%%%%%%%%%%%%%%%%%%%%%%%%%%%%%%%%%%%%%%%%%
\subsubsection{Construction de notre contre--exemple}
%%%%%%%%%%%%%%%%%%%%%%%%%%%%%%%%%%%%%%%%%%%%%%%%%%%%%%%%%%%%%%%%%%%%%%%%%%%%%%%%%%

Nous savons maintenant contr\^oler le seul changement conforme local qui pourrait rendre harmonique notre application C--harmonique pr\'ec\'edemment construite en d\'eformant la m\'etrique de d\'epart. On va montrer ici que les \'equations sont trop rigides, c'est--\`a--dire qu'il existe au moins une d\'eformation pour laquelle l'application et le changement conforme qui lui sont associ\'es ne v\'erifient pas la condition d'harmonicit\'e.

Soit $\tilde{Q}$ l'op\'erateur d\'efini de $\mathcal{M}^{k+4}$ dans $\Gamma^{k+1,\alpha}$ de la fa\c{c}on suivante,
\begin{equation}
 \tilde{Q}(g):=Q(\omega(g),\varphi(g),g,h),
\end{equation}
o\`u $g$ est une m\'etrique de $V'_e$, on se donne $\dot{g}$ dans $S^2TM$ et $(g_t)_{t\in[0,1]}$ une famille de m\'etriques de $W$ v\'erifiant le syst\`eme suivant\,:
 \begin{center}
 $\left\{\begin{array}{ll}
  g_0 & = h \\ \big[\partial_t g_t\big]_{t=0} & = \dot{g},
  \end{array}
 \right.$
\end{center}
nous allons montrer que si l'on choisit bien $\dot{g}$, il existe $s$ dans $[0,1]$ tel que la m\'etrique $g_s$ n'annule pas $\tilde{Q}$. Pour cela, on va calculer la diff\'erentielle ext\'erieure de la variation infinit\'esimale de $\tilde{Q}$ dans la direction $\dot{g}$ au point $(0,id,h,h)$ et montrer que celle--ci est non--nulle. Commen\c{c}ons par calculer la variation infinit\'esimale de $Q$ au point $(0,id,h,h)$, on obtient dans la direction $\dot{\omega}$\,:
\begin{equation*}
 \frac{\partial Q}{\partial \omega}(\dot{\omega}) = 2\, d\Delta\dot{\omega} - 4\ric d\dot{\omega},
\end{equation*}
dans la direction $\dot{\varphi}$ avec les formules (\ref{equation_P^2_phi}) et (\ref{equation_nablaT_phi})\,:
\begin{equation*}
 \frac{\partial Q}{\partial \varphi}(\dot{\varphi}) = \frac{2}{3} \scal (\delta d - \ric)\, \dot{\varphi} + 2\, \langle \ric,\nabla d\dot{\varphi} + \R_{\textstyle{\dot{\varphi},.}}. \rangle,
\end{equation*}
 et dans la direction $\dot{g}$ avec les formules (\ref{equation_P^2_g}), (\ref{equation_nablaT_g}) et \cite[1.174.e)]{MR867684}\,:
\begin{equation*}
 \frac{\partial Q}{\partial g}(\dot{g})
  = - \frac{1}{3} \scal (2\, \delta\dot{g} + d\tr\dot{g}) + \frac{1}{3}\, \big(d\Delta(\tr\dot{g}) + d\delta\delta\dot{g} - d\langle\ric, \dot{g}\rangle\big) - \langle \ric, 2\, \delta^*\!\dot{g} - \nabla \dot{g} \rangle,
\end{equation*}
o\`u les termes $\scal,\ric,\delta,\Delta$ ainsi que les traces et les produits scalaires sont donn\'es par rapport \`a $h$, et $\delta^*$ d\'esigne l'adjoint de la divergence $\delta$. Ce qui donne pour la variation infinit\'esimale de $\tilde{Q}$ au point $h$ dans la direction $\dot{g}$\,:
\begin{align*}
 T_h\tilde{Q}(\dot{g}) =
 & 2\, d\Delta T\omega(\dot{g}) - 4 \ric d\,T\omega(\dot{g}) + \frac{2}{3} \scal\, (\delta d - \Se)\, T\varphi(\dot{g}) \\
 & + 2\, \big\langle \ric,\nabla d\,T\varphi(\dot{g}) + \R_{\textstyle{\,T\varphi(\dot{g}),.}}. \big\rangle - \frac{1}{3} \scal\, (2\, \delta\dot{g} + d\tr\dot{g}) \\
 & + \frac{1}{3}\, \big(d\Delta(\tr\dot{g}) + d\delta\delta\dot{g} - d\langle\ric, \dot{g}\rangle\big) - \langle \ric, 2\,\delta^*\!\dot{g} - \nabla\dot{g} \rangle,
\end{align*}
comme $T\omega$ et $T\varphi$ sont respectivement d'ordre $0$ et $-1$ en $\dot{g}$, alors $T\tilde{Q}$ est donc d'ordre $3$ en $\dot{g}$. On calcule sa diff\'erentielle ext\'erieure, qu'on note $d\,T\tilde{Q}(\dot{g})$, pour supprimer les termes d'ordre $2$ et $3$, il ne restera que les termes d'ordre $1$ (qui seront des termes d'ordre $2$ dans la diff\'erentielle ext\'erieure)\,:
\begin{equation}
 \label{symbol_dTQ}
 d\,T\tilde{Q}(\dot{g}) = A(\dot{g}) + B(\dot{g}) + C(\dot{g}) +\, \mbox{des termes d'ordre inf\'erieurs},
\end{equation}
avec
\begin{align*}
 A(\dot{g}) &= -\, 4\, d\big(\ric d\,T\omega(\dot{g})\big), \\
 B(\dot{g}) &= d\langle \ric, 2\nabla d\,T\varphi(\dot{g}) - 2\, \delta^* \dot{g} + \nabla\dot{g} \rangle, \\
 C(\dot{g}) &= \frac{2}{3} \scal d\big(\delta d \, T\varphi(\dot{g}) - \delta\dot{g}\big).
\end{align*}
On va prouver que $d\,T\tilde{Q}$ n'est pas trivialement nul en montrant que son symbole n'est pas nul dans une certaine direction. Pour cela, on calcule le symbole de $T\varphi$ et de $T\omega$, on obtient au point $(id,h,h)$ avec les formules (\ref{equation_P^4_phi}) et (\ref{equation_P^4_g})\,:
\begin{equation*}
 \sigma_{T\varphi}(X) = \frac{1}{|X|^2}\, \dot{g}(X)-\frac{\tr \dot{g}}{6|X|^2}\, X - \frac{\dot{g}(X,X)}{3|X|^4}\, X
\end{equation*}
et ensuite au point $(0,id,h,h)$\,:
\begin{equation*}
 \sigma_{T\omega}(X) = - \frac{1}{6} \big(\tr\dot{g} - \frac{\dot{g}(X,X)}{|X|^2}\big).
\end{equation*}
On remarque que les symboles de $A$ et $B$ sont d'ordre $2$ si la m\'etrique $h$ n'est pas Einstein et que celui de $C$ est d'ordre inf\'erieur, plus pr\'ecis\'ement on a\,:
\begin{align*}
 \sigma_A (X) &= \frac{2}{3} \big(\tr\dot{g} - \frac{\dot{g}(X,X)}{|X|^2}\big)\, X \wedge \ric(X), \\
 \sigma_B (X) &= \frac{2\ric(X,X)}{|X|^2}\, X \wedge \dot{g}(X) - 2\, X \wedge \dot{g}(\ric(X)).
\end{align*}
Comme la m\'etrique $h$ ne satisfait pas la condition d'Einstein, il existe $Y$, $Z$ dans $TM$ et $\alpha$ un nombre r\'eel tels que $|Y|^2_h = 1$, $\ric(Y) = \alpha \, Y + Z$ et $h(Y,Z) = 0$.
Soit $\dot{g}$ une d\'eformation qui v\'erifie $\tr\dot{g} = 0$, $\dot{g}(Y) = Y$ et $\dot{g}(Z) = Z$, alors le symbole de $d\,T\tilde{Q}$ est non nul, puisque
\begin{equation*}
  \sigma_{d\,T\tilde{Q}}(Y) = \sigma_A (Y) + \sigma_B (Y) = -\frac{8}{3}\, Y \wedge Z .
\end{equation*}
Cela implique que $T\tilde{Q}$ n'est pas nul dans la direction $\dot{g}$, or $\tilde{Q}(g_0)=\tilde{Q}(h)=0$, donc il existe $g_s$ telle que $\tilde{Q}(g_s)$ ne soit pas nul, ce qui prouve que $\varphi(g_s)$ n'est pas harmonique pour aucun petit changement conforme \`a $g_s$.

%%%%%%%%%%%%%%%%%%%%%%%%%%%%%%%%%%%%%%%%%%%%%%%%% 
\subsubsection{Le changement conforme est local}
%%%%%%%%%%%%%%%%%%%%%%%%%%%%%%%%%%%%%%%%%%%%%%%%%

Supposons que $\varphi(g)$ est harmonique de $(M,e^{2f}g)$ dans $(M,h)$, alors on va montrer que si $g$ est suffisamment proche de $h$, alors on peut supposer que $f$ est petit.

\begin{lemme}
 \label{lemme_inf}
 Quelque soit $\lambda>0$, il existe un r\'eel $\mu$ strictement positif qui v\'erifie la propri\'et\'e suivante; quelque soit la m\'etrique $g$ v\'erifiant $\|g-h\|_{k+4,\alpha} < \mu$, alors s'il existe une fonction $f$ de classe $C^{k+4,\alpha}$ qui satisfait
 \begin{equation*}
  \Pa^2(\varphi(g),e^{2f} g,h) = 0,
 \end{equation*}
 alors il existe une fonction $\omega$ de classe $C^{k+4,\alpha}$ qui satisfait
 \begin{center}
  $\Pa^2(\varphi(g),e^{2\omega} g,h) = 0$ et $\|\omega\|_{k+4,\alpha} < \lambda$.
 \end{center}
\end{lemme}

\begin{proof}
 Soit $\lambda>0$ et $f$ une fonction $C^{k+4,\alpha}$, par continuit\'e il existe $\mu>0$ tel que pour toute m\'etrique $g$ v\'erifiant $\|g-h\|_{k+4,\alpha}<\mu$, alors
\begin{align*}
  \|df\|_{k+3,\alpha} - \|\langle df,T\varphi(g) \rangle_g\|_{k+3,\alpha} &\leq \frac{\lambda}{2\, \big(diam(M)+1\big)} \\
  \|\Pa^2(\varphi(g),g,h)\|_{k+3,\alpha}                         	  &\leq \frac{\lambda}{diam(M)+1},
 \end{align*}
 Si $f$ v\'erifie la condition du lemme alors $\Pa^2(\varphi(g),g,h) = 2\,\langle df,T\varphi(g) \rangle_g$, et on obtient avec ce qui pr\'ec\`ede\,:
 \begin{equation*}
  \label{e3}
  \|df\|_{k+3,\alpha} \leq \frac{\lambda}{diam(M)+1}.
 \end{equation*}
 Fixons $x_0$ un point de $M$ et posons $\omega:=f-f(x_0)$, d'apr\`es les in\'egalit\'es des accroissements finis, on a\,:
 \begin{equation*}
  \|\omega\|_{\infty} \leq diam(M)\, \|d\omega\|_{k+3,\alpha}.
 \end{equation*}
 D'autre part, $\|\omega\|_{k+4,\alpha} = \|\omega\|_{\infty} + \|d\omega\|_{k+3,\alpha}$, ainsi on obtient
 \begin{equation*}
  \|\omega\|_{k+4,\alpha} \leq \big(1+diam(M)\big)\, \|d\omega\|_{k+3,\alpha},
 \end{equation*}
 ce qui cl\^ot la preuve, vu que $df=d\omega$ et que $\Pa^2(\varphi(g),e^{2\omega} g,h) = 0$.
\end{proof}

%%%%%%%%%%%%%%%%%%%%%%%%%%%%%%%%%%%%%%%%%%%%%%
\subsubsection{Conclusion}
%%%%%%%%%%%%%%%%%%%%%%%%%%%%%%%%%%%%%%%%%%%%%%

On reprend les notations du lemme \ref{lemme_def_omega}. On a un voisinage $V'_e$ de $g_e$ dans $\mathcal{M}^{k+4,\alpha}$ et un voisinage $V_0$ de la fonction nulle dans $C^{k+4,\alpha}$. Il existe alors deux nombres r\'eels $\lambda$ et $R$, strictement positifs, et tels que la boule $B_{\lambda}$ dans $C^{k+4,\alpha}$ de centre la fonction nulle et de rayon $\lambda$ et la boule $B_R$ dans $\mathcal{M}^{k+4,\alpha}$ de centre $g_e$ et de rayon $R$ v\'erifient $\omega(B_R) \subset B_{\lambda} \subset V_0$.
Soit $\varepsilon < R$ un nombre r\'eel strictement positif et $h$ une m\'etrique fix\'ee de $V'_e$, non Einstein, de courbure scalaire constante \'egale \`a celle de $g_e$ et vérifiant
\begin{equation*}
 \|h-g_e\|_{k+4,\alpha} < \frac{\varepsilon}{2}.
\end{equation*}

On se donne une m\'etrique $g$ qui v\'erifie $\|g-h\|_{k+4,\alpha} < \min(\mu,\epsilon /2)$, o\`u $\mu$ est d\'efini par le lemme \ref{lemme_inf} et on suppose qu'il existe un changement conforme $f$ tel que l'application $\varphi(g)$ est harmonique de $(M,e^{2f}g)$ dans $(N,h)$. La m\'etrique $g$ v\'erifie
\begin{equation*}
 \|g-g_e\|_{k+4,\alpha} < \|g-h\|_{k+4,\alpha} + \|h-g_e\|_{k+4,\alpha} < \varepsilon
\end{equation*}
et comme $\|g-h\|_{k+4,\alpha} < \mu$, alors le changement conforme $\omega$ du lemme \ref{lemme_inf} est dans $B_{\lambda}$. De plus la m\'etrique $g$ est dans $B_R$, alors $\omega=\omega(g)$ \`a une constante pr\`es d'apr\`es le lemme \ref{lemme_def_omega}. On vient donc de montrer que pour de telles m\'etriques $g$, le probl\`eme global se r\'esume au probl\`eme local.

%%%%%%%%%%%%%%%%%%%%%%%%%%%%%%%%%%%%%%%%%%%%%%%%%
\section{Formules de variations au premier ordre} 
%%%%%%%%%%%%%%%%%%%%%%%%%%%%%%%%%%%%%%%%%%%%%%%%%

Dans la suite on utilisera la convention suivante; on va indicer par $h$ les termes qui se r\'ef\`erent \`a $h$ et ne rien mettre pour ceux qui se r\'ef\`erent \`a $g$. 
\begin{proposition}
 Avec les notations pr\'ec\'edentes, on obtient pour les d\'eformations infinit\'esimales par rapport aux applications au point $(0,\varphi,g,h)$\,:
 \begin{align}
  \label{equation_nablaT_phi}
  \frac{\partial \nabla T}{\partial \varphi} (\dot{\varphi}) &= \nabla d \dot{\varphi} + \R^h_{\,\textstyle{\dot{\varphi},T\varphi}} T\varphi \\
  \label{equation_P^2_phi}
  \frac{\partial \Pa^2}{\partial \varphi} (\dot{\varphi})    &= (\delta d - \Se^h) (\dot{\varphi}),
 \end{align}
 et par rapport \`a la m\'etrique de d\'epart au point $(0,\varphi,g,h)$\,:
 \begin{align}
  \label{equation_nablaT_g}
  \frac{\partial \nabla T\varphi}{\partial g} (\dot{g}) &= - \langle T\varphi , \delta^*\dot{g} - \frac{1}{2}\, \nabla \dot{g} \rangle \\
  \label{equation_P^2_g}
  \frac{\partial \Pa^2}{\partial g} (\dot{g})           &= - \delta\big(\dot{g}\, (T\varphi)\big) - \frac{1}{2}\, \langle d\tr\dot{g},T\varphi \rangle,
 \end{align}
 o\`u $\delta^*\dot{g}$ est d\'efini de la mani\`ere suivante\,:
 \begin{equation*}
  \delta^*\dot{g}(X,Y,Z) := \frac{1}{2}\, \big(\nabla_{\textstyle{X}} \dot{g} (Y,Z) + \nabla_{\textstyle{Y}} \dot{g} (X,Z)\big).
 \end{equation*}
\end{proposition}

\begin{proof}
 On note $(\varphi_t)_{t\in[0,1]}$ une famille \`a $1$--param\`etre d'applications de $M$ dans $N$, v\'erifiant le syst\`eme suivant\,:
  \begin{center}
  $\left\{
   \begin{array}{ll}
   \varphi_0 & = \varphi \\
   \big[\partial_t\varphi_t\big]_{t=0} & = \dot{\varphi}.
   \end{array}
  \right.$
 \end{center}
 On munit $M\times[0,1]$ de la m\'etrique $\gamma=g + dt^2$ et on pose $\varPhi$ l'application de $M\times[0,1]$ dans $N$ d\'efinie par $\varPhi(p,t)=\varphi_t(p), \forall (p,t)\in M\times[0,1]$. On note $\nabla^{\gamma,h}$ la connexion de Levi--Civita du fibr\'e $\Omega(M)\otimes\varPhi^*TN$ et on se donne deux vecteurs $X$ et $Y$ de $TM$, alors\,:
\begin{align*}
 \lefteqn{\nabla^h_{\!\textstyle{\partial_t\varphi_t}} \big((\nabla^{\gamma,h}_{\textstyle{X}} T\varphi_t) \,Y\big)} \\
  &= \nabla^h_{\textstyle{T\varphi_t(X)}} \big((\nabla^{\gamma,h}_{\textstyle{\partial_t}} T\varphi_t) \,Y\big) + \R^h_{\,\textstyle{\partial_t\varphi_t,T\varphi_t(X)}} T\varphi_t(Y) - \nabla^h_{\,\textstyle{\partial_t\varphi_t}} \big(T\varphi_t(\nabla^g_{\,\textstyle{X}} Y)\big)\\
  &= \nabla^h_{\textstyle{T\varphi_t(X)}} \big((\nabla^{\gamma,h}_{\textstyle{Y}} T\varphi_t) \partial_t\big) + \R^h_{\,\textstyle{\partial_t\varphi_t,T\varphi_t(X)}} T\varphi_t(Y) - \nabla^h_{\,\textstyle{\partial_t\varphi_t}} \big(T\varphi_t(\nabla^g_{\,\textstyle{X}} Y)\big),
\end{align*}
 ce qui donne bien les deux premi\`eres \'egalit\'es. Les deux derni\`eres se montrent au moyen de la formule \cite[1.174.a)]{MR867684}.
\end{proof}

\begin{proposition}
Avec les notations pr\'ec\'edentes, $n=4$ et en supposant que la courbure scalaire de $h$ est constante, on obtient au point $(id,h,h)$\,:
 \begin{align}
 \label{equation_P^4_phi}
 \frac{\partial \Pa^4}{\partial \varphi} (\dot{\varphi})
  =& (\delta d+\frac{2}{3}\scal-\ric)\, (\delta d-\ric)\, \dot{\varphi} + 2\, \langle \ric,\nabla d\dot{\varphi} + \R_{\textstyle{\dot{\varphi},T\varphi}}T\varphi \rangle \\
 \label{equation_P^4_g}
 \frac{\partial \Pa^4}{\partial g} (\dot{g})
  =& - (\delta d+\frac{2}{3}\scal-\ric) (\delta\dot{g} + \frac{1}{2}\, d\tr\dot{g}) + \frac{1}{3}\, (d\,\Delta\tr\dot{g} + d\delta\delta\dot{g}) \notag\\
  &  - \frac{1}{3}\, d\langle \ric, \dot{g} \rangle - \langle \ric, 2\, \delta^*\dot{g} - \nabla\dot{g} \rangle.
 \end{align}
\end{proposition}

\begin{proof}
 Avec l'identit\'e de Bianchi, on obtient
 \begin{equation*}
  \Pa^4(\varphi,g,h) = (\delta d+\frac{2}{3}\scal-\Se^h) \Pa^2(0,\varphi,g,h) + \frac{1}{3} \langle d\scal,T\varphi \rangle_g + 2 \langle \ric,\nabla T\varphi \rangle_g,
\end{equation*}
 on montre alors facilement la premi\`ere \'egalit\'e avec (\ref{equation_P^2_phi}) et (\ref{equation_nablaT_phi}), et la deuxi\`eme avec (\ref{equation_P^2_g}), (\cite[1.174.e)]{MR867684}) et (\ref{equation_nablaT_g}).
\end{proof}
 
\bibliography{biblio}
\bibliographystyle{plain}

\end{document}